\documentclass[graybox]{svmult}

\usepackage{amsfonts,amsmath, amssymb,amscd,mathtools,mathrsfs,xspace,stmaryrd,graphicx,wasysym,enumerate}
\usepackage[all]{xy}

\usepackage{mathptmx}       
\usepackage{helvet}         
\usepackage{courier}        
\usepackage{type1cm}        

\usepackage{makeidx}         
\usepackage{graphicx}        
\usepackage{multicol}        
\usepackage[bottom]{footmisc}

\usepackage{color}

\allowdisplaybreaks[4]

\def\un{\hbox{\bf 1}}

\def\shuffle{{\scriptscriptstyle \;\sqcup \hspace*{-0.07cm}\sqcup\;}}
\usepackage{xspace}
\newcommand{\id}{\mathrm{id}}

\definecolor{Light}{gray}{0.85}

\def\abs#1{\left\vert #1 \right\vert}
\def\allpoly{\mbox{$\re\langle X \rangle$}}

\def\allpolyx0degn{\mbox{$P_n$}}

\def\allseries{\mbox{$\re\langle\langle X \rangle\rangle$}}

\def\allseriesm{\mbox{$\re^m\langle\langle X \rangle\rangle$}}

\def\allseriesmtimesm{\mbox{$\re^{m\times m}\langle\langle X \rangle\rangle$}}
\def\allseriesmminusone{\mbox{$\re^{m-1}\langle\langle X \rangle\rangle$}}

\newcommand{\comment}[1]{} 

\def\doubleone{{\rm\, l\!l}}

\def\Endallseries{{\rm End}(\allseries)}

\def\eqref#1{(\ref{#1})} 

\def\id{{\rm id}}

\def\modcomp{\:\tilde{\circ}\,} 

\def\norm#1{\left\Vert#1\right\Vert}

\def\re{{\mathbb R}} 

\def\shuffle{{\scriptscriptstyle \;\sqcup \hspace*{-0.05cm}\sqcup\;}}

\def\begals{\[\begin{aligned}}
\def\endals{\end{aligned}\]}
\def\begce{\begin{center}}
\def\endce{\end{center}}
\def\begar{\begin{array}}
\def\endar{\end{array}}
\def\begeq{\begin{equation}}
\def\endeq{\end{equation}}
\def\begdi{\begin{displaymath}}
\def\enddi{\end{displaymath}}
\def\begdis{\begin{eqnarray*}}
\def\enddis{\end{eqnarray*}}
\def\begeqa{\begin{eqnarray}}
\def\endeqa{\end{eqnarray}}
\def\begdes{\begin{description}}
\def\enddes{\end{description}}
\def\begit{\begin{itemize}}
\def\endit{\end{itemize}}
\def\begen{\begin{enumerate}}
\def\enden{\end{enumerate}}
\def\beglar{\left[\begin{array}}
\def\endrar{\end{array}\right]}
\def\begle{\begin{lemma}}
\def\endle{\end{lemma}}
\def\begde{\begin{definition}}
\def\endde{\end{definition}}
\def\begth{\begin{theorem}}
\def\endth{\end{theorem}}
\def\begco{\begin{corollary}}
\def\endco{\end{corollary}}
\def\begprop{\begin{proposition}}
\def\endprop{\end{proposition}}
\def\begex{\begin{example}}
\def\endex{\end{example}}
\def\begexer{\begin{exercise}}
\def\endexer{\end{exercise}}

\def\begres{\noindent{\bf Remarks}:\begin{enumerate}}
\def\endres{\end{enumerate} \par}
\def\begpr{\begin{proof}}
\def\endpr{\end{proof}}
\def\begtab{\begin{tabular}}
\def\endtab{\end{tabular}}
\def\rref#1{(\ref{#1})}
\def\lbracedef{\left\{\begin{array}{@{\hspace*{2pt}}c@{\hspace*{3pt}}c@{\hspace*{3pt}}l}} 

\makeindex             

\begin{document}

\title*{The Fa\`{a} di Bruno Hopf algebra for multivariable feedback recursions in the center problem for higher order Abel equations}
\titlerunning{Fa\`{a} di Bruno Hopf algebra, center problem, higher order Abel equations}

\author{Kurusch Ebrahimi-Fard, W.~Steven~Gray}
\authorrunning{Fa\`{a} di Bruno Hopf algebra, center problem, higher order Abel equations}

\institute{
Kurusch Ebrahimi-Fard \at Norwegian University of Science and Technology, 7491 Trondheim, Norway,\\
On leave from UHA, Mulhouse, France. \email{kurusch.ebrahimi-fard@ntnu.no}.
\and
W.~Steven~Gray \at Old Dominion University,
Norfolk, Virginia 23529 USA,
\email{sgray@odu.edu}}

\maketitle


\abstract{
Poincar\'{e}'s center problem asks for conditions under which a planar polynomial system of ordinary differential equations has a center. It is well understood that the Abel equation naturally describes the problem in a convenient coordinate system. In 1989, Devlin described an algebraic approach for constructing sufficient conditions for a center using a linear recursion for the generating series of the solution to the Abel equation. Subsequent work by the authors linked this recursion to feedback structures in control theory and combinatorial Hopf algebras, but only for the lowest degree case. The present work introduces what turns out to be the nontrivial multivariable generalization of this connection between the center problem, feedback control, and combinatorial Hopf algebras. Once the picture is completed, it is possible to provide generalizations of some known identities involving the Abel generating series. A linear recursion for the antipode of this new Hopf algebra is also developed using coderivations. Finally, the results are used to further explore what is called the composition condition for the center problem.
}

\abstract*{
Poincar\'{e}'s center problem asks for conditions under which a planar polynomial system of ordinary differential equations has a center. It is well understood that the Abel equation naturally describes the problem in a convenient coordinate system. In 1989, Devlin described an algebraic approach for constructing sufficient conditions for a center using a linear recursion for the generating series of the solution to the Abel equation. Subsequent work by the authors linked this recursion to feedback structures in control theory and combinatorial Hopf algebras, but only for the lowest degree case. The present work introduces what turns out to be the nontrivial multivariable generalization of this connection between the center problem, feedback control, and combinatorial Hopf algebras. Once the picture is completed, it is possible to provide generalizations of some known identities involving the Abel generating series. A linear recursion for the antipode of this new Hopf algebra is also developed using coderivations. Finally, the results are used to further explore what is called the composition condition for the center problem.
}

\noindent {\footnotesize{}}



\section{Introduction}
\label{sect:intro}

The classical center problem first studied by Henri Poincar\'{e} \cite{Poincare_1928} considers a system of planar ordinary differential equations
\begin{equation}
\label{eq:Poincare-problem}
	\frac{dx}{dt}=X(x,y),\qquad \frac{dy}{dt}=Y(x,y),
\end{equation}
where $X,Y$ are homogeneous polynomials with a linear part of center type.
The equilibrium at the origin is a center if it is contained in an open neighborhood $U$ having no other equilibria,
and every trajectory of system (\ref{eq:Poincare-problem}) in $U$ is closed with the same period $\omega$.
The problem is usually studied in its canonical form via a reparametrization that transforms
(\ref{eq:Poincare-problem}) into the Abel equation
\begin{equation}
\label{Devlineq}
		\dot{z}(t) = v_1(t)z^2(t) + v_2(t)z^3(t),
\end{equation}
where $v_1$ and $v_2$ are continuous real-valued functions \cite{Alwash-Lloyd_87,Cherkas_76,Lloyd_82}. In this setting, the origin $z=0$ is a center if $z(0)=z(\omega)=r$ for $r>0$ sufficiently small and $\omega>0$ fixed. The center problem is to determine the largest class of functions $v_1$ and $v_2$ that will render $z=0$ a center.

An algebraic approach to the  center problem was first proposed by Devlin in 1989 \cite{Devlin_1989,Devlin_1991},
which was based on the work of Alwash and Lloyd \cite{Alwash-Lloyd_87,Lloyd_82}. In modern parlance, Devlin's method was to first write the solution of the Abel equation \rref{Devlineq} in terms of a {\em Chen--Fliess functional expansion} or {\em Fliess operator} \cite{Fliess_81,Fliess_83} whose coefficients are parameterized by $r$. A Fliess operator is simply a weighted sum of iterated integrals of $v_1$ and $v_2$ indexed by words in the noncommuting symbols $x_1$ and $x_2$, respectively.
The concept is widely used, for example, in control theory to describe the input-output map
of a system modeled in terms of ordinary differential equations. (For readers not familiar with this subject, the following references provide
a good overview \cite{Fliess_81, Fliess_83,Isidori_95,Kawski-Sussmann_97,Nijmeijer-vanderSchaft_90,Wang-Sontag_92,Wang-Sontag_92a,Wang-Sontag_95,Wang_90,Wang_95}.) 
Devlin showed that the generating series for his particular Fliess operator
with $r=1$, which is a formal power series $c_A$ over words in the alphabet $X=\{x_1,x_2\}$, can be decomposed as
\begeq \label{eq:Devlin-decomposition}
	c_A=\sum_{n=1}^\infty c_A(n),
\endeq
where the polynomials $c_A(n)$, $n \geq 1$ satisfy the linear recursion
\begin{equation} \label{eq:Devlin-recursion}
	c_A(n)=(n-1)c_A(n-1)x_1+(n-2)c_A(n-2)x_2,\quad n\geq 2
\end{equation}
with $c_A(1)=1$ and $c_A(0)=0$.
Here $\deg(x_i):=i$, and each letter $x_i$ encodes the contribution of $\nu_i$ to the series solution of $\rref{Devlineq}$.
His derivation used the underlying shuffle algebra induced by products of iterated integrals
rather than the fact that the operator coefficients are differentially generated from the vector fields in the Abel
equation \rref{Devlineq} \cite{Fliess_81,Isidori_95,Nijmeijer-vanderSchaft_90}. Devlin also provided a recursion for the higher-order Abel equation
\begeq \label{eq:Abel-eqn-degree-m}
	\dot{z}(t)=\sum_{i=1}^m v_i(t)z^{i+1}(t),\quad m\geq 2,
\endeq
though the calculations become somewhat intractable.
Using such recursions, it was then possible to synthesize various sufficient conditions on the  $v_i$  under which the origin was a center. This included a generalization of the {\em composition condition} of \cite{Alwash-Lloyd_87}. The latter states that a sufficient condition for a center is the existence of a differentiable function $q$ such that $q(\omega)=q(0)$ for some $\omega>0$ and
\begeq \label{eq:composition-condition}
	v_i(t)=\bar{v}_i(q(t)) \dot{q}(t),\quad i=1,\ldots,m,
\endeq
where the $\bar{v}_i$  are continuous functions. For a time it was conjectured that this condition was also a necessary condition for a center if certain constraints were imposed on the $v_i$, for example, if they were polynomial functions of $\cos \omega t$ and $\sin \omega t$. However, a counterexample to this claim was later given by Alwash in \cite{Alwash_89}. It is still believed, however, to be a necessary condition when the $v_i$ are polynomials. This is now called the {\em composition conjecture} (see \cite{Alwash_09,Briskin-etal_10,Brinskin-Yomdin_05,Brudnyi_10,Yomdin_03} and the references in the survey article~\cite{Gine-etal_16}).

\begin{figure}[tb]
\begin{center}
\includegraphics[scale=0.35]{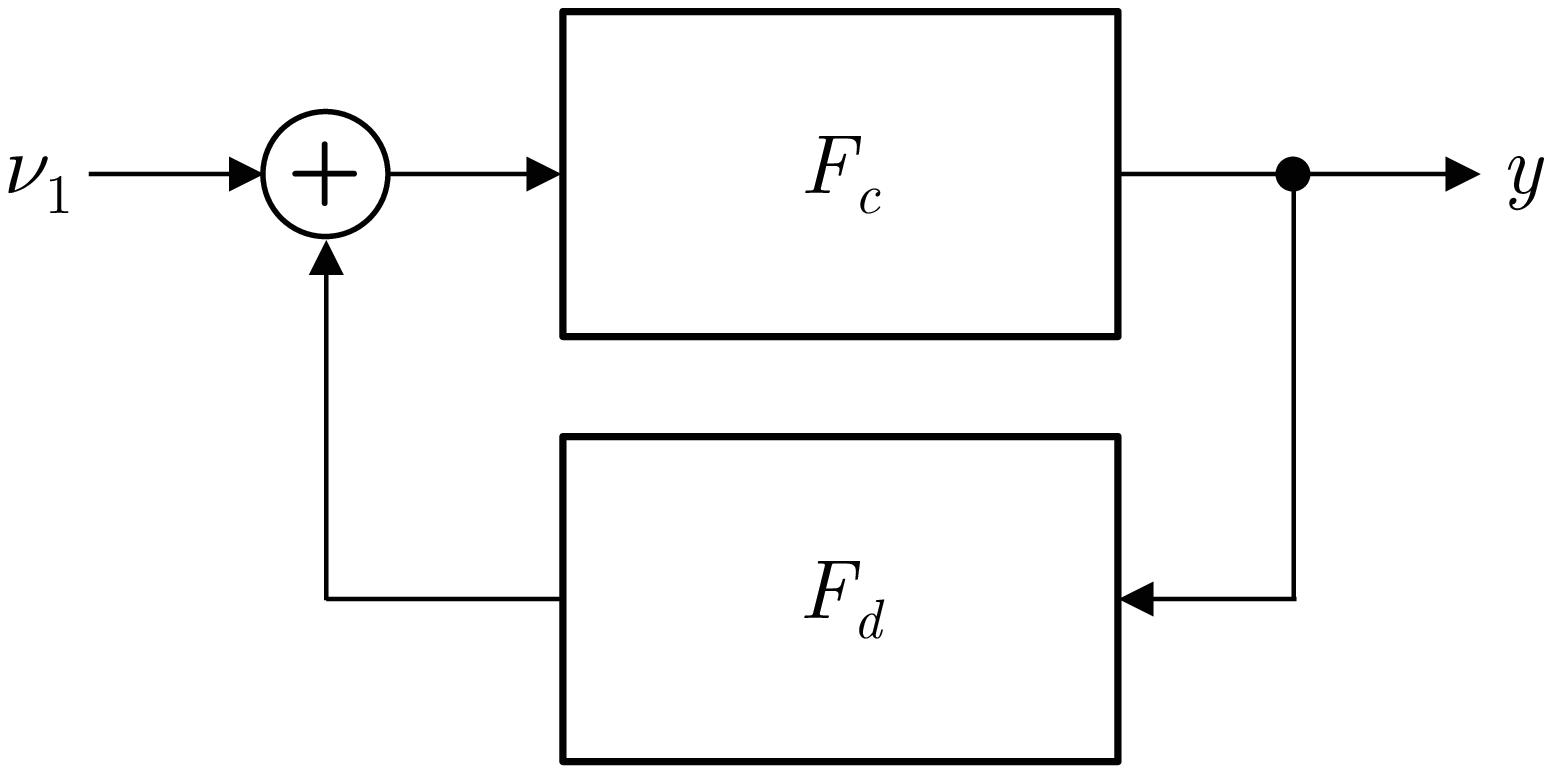}
\end{center}
\caption{Feedback connection of Fliess operators $F_c$ and $F_d$}
\label{fig:output-feedback}
\end{figure}

Recently, the authors revisited Devlin's method in a combinatorial Hopf algebra setting in light of the fact that the Abel equation was found to play a central role in determining the radius of convergence of feedback connected Fliess operators as shown in Figure~\ref{fig:output-feedback} \cite{Thitsa-Gray_12}.
This recursive structure is described by the feedback equation
\begdi
y(t)=F_c[\nu_1(t)+F_d[y(t)]],
\enddi
which by a suitable choice of generating series $c$ and $d$ involving an arbitrary function $\nu_2(t)$ can be written
directly in the form
\begin{align*}
\dot{z}(t)&=z^2(t)[\nu_1(t)+\nu_2(t)y(t)] \\
&=\nu_1(t)z^2(t)+\nu_2(t)z^3(t),
\end{align*}
where $y(t)=z(t)$.
It was shown in \cite{Ebrahimi-Fard-Gray_IMRN17} that the decomposition \rref{eq:Devlin-decomposition} is exactly the sum of the graded components of a Hopf algebra antipode applied to the formal power series $-c_F$, where
\begin{equation}
\label{FFsystem}
	c_F=\sum_{k = 0}^\infty k!\, x_1^k
\end{equation}
is the Ferfera series, that is, the generating series for solution of the equation $\dot{z}=z^2u$, $z(0)=1$ \cite{Ferfera_79,Ferfera_80}. The link is made using the Hopf algebra of output feedback which encodes the {\em composition} of iterated integrals rather than their products \cite{DuffautEFG_2014,Gray-Duffaut_Espinosa_SCL11,Gray-et-al_SCL14}. As a consequence, another algebraic structure at play in Devlin's approach beyond the shuffle algebra is a Fa\`{a} di Bruno type Hopf algebra. Now it is a standard theorem that the antipode of every connected graded Hopf algebra can be computed recursively \cite{Figueroa-Gracia-Bondia_05,manchon2}. This fact was exploited, for example, in the authors' application of the output feedback Hopf algebra to compute the feedback product, a device used to
compute the generating series for Fliess operator representation of the interconnection shown in Figure~\ref{fig:output-feedback} \cite{DuffautEFG_2014,Gray-et-al_MTNS14}. But somewhat surprisingly it was also shown in \cite{Ebrahimi-Fard-Gray_IMRN17} that for this Hopf algebra the antipode could be computed {\em in general} using a linear recursion of Devlin type. This method has been shown empirically to be more efficient than all existing methods for computing the antipode \cite{Berlin-etal_CISS17}, which is useful in control applications \cite{Duffaut_Espinosa-Gray_ACC17,Gray-Duffaut_Espinosa_FdB14,Gray-et-al_CDC15,Gray-et-al_Auto14}. What was not evident, however, was how all of these ideas could be related for higher order Abel equations, i.e., equation \rref{eq:Abel-eqn-degree-m} when~$m>2$.

The goal of this paper is to present what turns out to be the nontrivial generalization of the connection between the center problem, control theory, and combinatorial Hopf algebras for higher order Abel equations. It requires a new class of matrix-valued Fliess operators with a certain Toeplitz structure in order to provide the proper grading. In addition, a new type of multivariable output feedback Hopf algebra is needed, one which is distinct from that described in \cite{DuffautEFG_2014,Gray-Duffaut_Espinosa_SCL11,Gray-et-al_SCL14} and is more closely related to the {\em output affine feedback} Hopf algebra introduced in \cite{Gray-Ebrahimi-Fard_SIAMzz} for the $m=2$ case with $v_2=1$ (so effectively the single-input--single-output case) to describe {\em multiplicative} output feedback. Once the picture is completed, it is possible to provide higher order extensions of some known identities for the Abel generating series, $c_A$. A linear recursion for the antipode of this new Hopf algebra is also developed using coderivations. Finally, a new sufficient condition for a center is given inspired
by viewing the Abel equation in terms of a feedback condition. This in turn provides another way of interpreting the composition condition.

%


\section{Linear recursions for differentially generated series and their inverses}
\label{sect:multiAbel}

The starting point is to show how any formal power series whose coefficients are differentially generated by a set of analytic vector fields can be written in terms of a linear recursion, as can its inverse in a certain compositional sense. This implicitly describes a group that will be utilized in the next section to describe recursions derived from feedback systems.

Consider the set of formal power series $\allseries$ over the set of words $X^\ast$ generated by an alphabet of noncommuting symbols $X=\{x_1,\ldots,x_m\}$. Elements of $X$ are called {\em letters}, and {\em words} over $X$ consist of finite sequences of letters, $\eta=x_{i_1}\cdots x_{i_k} \in X^*$. The {\em length} of a word $\eta$ is denoted $\abs{\eta}$ and is equivalent to the number of letters it contains. When viewed as a graded vector space, where $\deg(x_i):=i$ and $\deg(\mathrm{e}):=0$ with $\mathrm{e}$ denoting the empty word $\emptyset\in X^*$, any $c \in \allseries$ can be uniquely decomposed into its homogeneous components $c=\sum_{n\geq 1} c(n)$ with $\deg(c(n))=n-1$, $n\geq 1$. In particular, if $X^k$ is the set of all words of length $k$, then $c(1)=\langle c,\mathrm{e}\rangle\mathrm{e}$ and
\begeq \label{eq:cn-components}
	c(n)=\sum_{i=1}^{\min(m,n-1)}\sum_{\eta\in X^{n-1-i}} \langle c,\eta x_i\rangle\, \eta x_i,	\quad n\geq 2.
\endeq
A series $c \in \allseries$ is said to be {\em differentially generated} if there exists a set of analytic vector fields $\{g_1,g_2,\ldots,g_m\}$
defined on a neighborhood $W$ of $z_0\in\re^n$ and an analytic function $h:W \to \re$ such that for every word $\eta$ in $X^\ast$
the corresponding coefficient of $c$ can be written as
\begin{displaymath}
	\langle c,\eta \rangle=L_{g_{j_1}}\cdots L_{g_{j_k}}h(z_0), \quad \eta=x_{j_k}\cdots x_{j_1},
\end{displaymath}
where the Lie derivative of $h$ with respect to $g_j$ is defined as the linear operator
\begin{displaymath}
	L_{g_j} h: W\rightarrow \re,\quad z \mapsto L_{g_j} h(z):=\frac{\partial h}{\partial z}(z)\, g_j(z).
\end{displaymath}
The tuple $(g_1,g_2,\ldots,g_m,z_0,h)$ will be referred to as the {\em generator} of $c$.
It follows directly that $c(n)=P_{n-1}(z_0)$, where $P_{0}(z_0)=h(z_0)\mathrm{e}$ and for $n>0$, $P_n(z):=\sum_{\eta\in X^n} L_{g_\eta}h(z)\,\eta$, with $L_{g_\eta}:=L_{g_{j_1}}\cdots L_{g_{j_k}}$, and \rref{eq:cn-components} can be rewritten as the linear recursion
\begeq \label{eq:Pn-recursion}
	P_n(z_0)=\sum_{i=1}^{\min(m,n)} L_{g_i}P_{n-i}(z_0)x_i, \quad n\geq 1.
\endeq
In this case the grading on $\allseries$ can be encoded in the sequence $P_n(z_0)$, $n\geq 1$, by assigning degrees to the vector fields, namely, $\deg(g_i):=\deg(x_i)=i$, $i=1,\ldots,m$.

\begin{example} \label{ex:Ferfera-recursion}
Suppose $m=1$, $g_1(z)=z^2$, $z_0=1$, and $h(z)=z$. Then $c(1)=P_0(1)=h(1)\mathrm{e}=\mathrm{e}$ and
\begdi
	c(n)=P_{n-1}(1)=L_{z^2}P_{n-2}(1)x_1
	      =(n-1)P_{n-2}(1)x_1=(n-1)c(n-1)x_1,	\quad n\geq 2.
\enddi
In which case,
\begdi
	\sum_{n=1}^\infty c(n)
	      =\sum_{n=1}^\infty (n-1)!\,x_1^{n-1}
	      =\sum_{n=0}^\infty n!\,x_1^n=:c_F. \label{eq:c_F}
\enddi
This is the well studied generating series of Ferfera \cite{Ferfera_79,Ferfera_80}.
\end{example}

Now suppose $d\in\allseries$ is differentially generated, and consider the corresponding {\it{Chen--Fliess series}} or {\it{Fliess
operator}}
\begdi
	F_d[u](t) :=\sum_{\eta\in X^{\ast}} \langle d,\eta \rangle\,E_\eta[u](t,t_0), 
\enddi
where $E_\eta[u]$ is defined inductively for each word $\eta \in X^{\ast}$ as an iterated integral over the {\it{controls}} $u:=(u_1(t),\ldots,u_m(t))$, $u_i: [t_0,t] \rightarrow\re$, by $E_\emptyset[u]:=1$ and
\begdi
	E_{x_i\bar{\eta}}[u](t,t_0) := \int_{t_0}^tu_{i}(\tau)E_{\bar{\eta}}[u](\tau,t_0)\,d\tau
\enddi
with $x_i\in X$, $\bar{\eta}\in X^{\ast}$.
If $u \in L_1^m[t_0,t_0+T]$, that is, $u$ is measurable with finite $L_{1}$-norm, $\norm{u}_{L_1}:=\max\{\norm{u_i}_{1}: \ 1\le
i\le m\}<R$, then the analyticity of the generator for $d$ is sufficient to guarantee that the Fliess operator $F_d[u](t)$ converges absolutely and uniformly on $[0,T]$ for sufficiently small $R,T>0$ \cite{Gray-Wang_SCL02}. Suppose next that $d=(d_1,\ldots,d_{m-1})$ is a family of series $d_i\in\allseries$, $i=1,\ldots,m-1$ which are differentially generated by $(g_1,\ldots,g_m,z_0,h_1,\ldots,h_{m-1})$, and define the associated {\it{Toeplitz matrix}}
\begdi
	d_{{\rm{Toep}}}:=
	\left[
		\begin{array}{ccccc}
		1 & d_1 & d_2 & \cdots & d_{m-1} \\
		0 & 1   & d_1 & \cdots & d_{m-2} \\
	 \vdots  & \vdots  & \vdots & \vdots & \vdots \\
		0 & 0 & 0 & \cdots & d_1 \\
		0 & 0   & 0 & \cdots & 1
		\end{array}
	\right]
	=I+\sum_{i=1}^{m-1} d_i N^i,
\enddi
where $I\in\re^{m\times m}$ is the identity matrix, and $N\in\re^{m\times m}$ is the nilpotent matrix consisting of zero entries except for a super diagonal of ones. The {\em Toeplitz affine Fliess operator} is taken to be $y=F_{d_\delta}[u]:=F_{d_{{\rm{Toep}}}}[u]u$, which can be written in expanded form~as
\begdi
	\left[
	\begin{array}{c}
		y_1\\
		y_2 \\
		\vdots\\
		y_{m-1} \\
		y_m
	\end{array}
	\right]=
	\left[
	\begin{array}{ccccc}
	1 & F_{d_1}[u] & F_{d_2}[u] & \cdots & F_{d_{m-1}}[u] \\
	0 & 1   & F_{d_1}[u] & \cdots & F_{d_{m-2}}[u] \\
	\vdots  & \vdots  & \vdots & \vdots & \vdots \\
	0 & 0 & 0 & \cdots & F_{d_1}[u] \\
	0 & 0   & 0 & \cdots & 1
	\end{array}
	\right]
	\left[
	\begin{array}{c}
	u_1\\
	u_2 \\
	\vdots\\
	u_{m-1} \\
	u_m
	\end{array}
	\right].
\enddi
Note in particular that $0_{{\rm{Toep}}}=I$ so that $F_{0_\delta}[u]=u$. The operator $F_{d_\delta}$ is {\em realized} by the analytic state space system
\begin{subequations}
	\begin{align}
	\dot{z}&= \sum_{i=1}^m g_i(z)u_i, \quad z(0)=z_0 \label{eq:state-equation} \\
		y&=H(z)u,
\end{align}
\end{subequations}
where
\begeq \label{eq:Toeplitz-matrix}
	H=
	\left[
	\begin{array}{ccccc}
	1 & h_1 & h_2 & \cdots & h_{m-1} \\
	0 & 1   & h_1 & \cdots & h_{m-2} \\
	\vdots  & \vdots  & \vdots & \vdots & \vdots \\
	0 & 0 & 0 & \cdots & h_1 \\
	0 & 0   & 0 & \cdots & 1
	\end{array}
	\right]
	=I+\sum_{i=1}^{m-1} h_i N^i,
\endeq
in the sense that on some neighborhood $W$ of $z_0$, \rref{eq:state-equation} has a well defined solution $z(t)$ on $[t_0,t_0+T]$ and $y=F_{d_{{\rm{Toep}}}}[u]u=H(z)u$ on this same interval. Since the Toeplitz matrix $H$ is always invertible and Toeplitz, it follows that the inverse operator $u=F_{d_\delta^{-1}}[y]:=F_{d_{{\rm{Toep}}}^{-1}}[y]y$ is another Toeplitz affine Fliess operator realized by the state space system
\begin{subequations} \label{eq:inverse-affine-FO}
\begin{align}
	\dot{z}&=\sum_{i=1}^m g_i(z)[H^{-1}(z)y]_i,\quad z(0)=z_0 \\
		u&=H^{-1}(z)y.
\end{align}
\end{subequations}
so that $F_{d_\delta}\circ F_{d_\delta^{-1}}=F_{d_\delta^{-1}}\circ F_{d_\delta}=I$. (Here $[y]_i$ denotes the $i$ component of $y\in\re^m$.) The generating series for the inverse operator, $d^{-1}=(d^{-1}_1,\ldots,d^{-1}_{m-1})$, is differentially generated by $(\tilde{g}_1,\ldots,\tilde{g}_m,z_0,\tilde{h}_1,\ldots,\tilde{h}_{m-1})$, where $\tilde{g}_i:= \sum_{j=1}^m g_jH^{-1}_{ji}=g_i + \sum_{j=1}^{i-1} g_{i-j}\tilde{h}_j$ with $\tilde{h}_j:=H^{-1}_{1,1+j}$.

\begin{example} \label{ex:mequal3-Toeplitz-affine-inverse-realization}
For the case where $m=3$, system (\ref{eq:inverse-affine-FO}) becomes
\begin{subequations} \label{eq:mequal3-Toeplitz-affine-inverse-realization}
\begin{align}
	\dot{z}&= g_1y_1+(g_2-g_1h_1)y_2+(g_3-g_2h_1+g_1(h_1^2-h_2))y_3, \quad z(0)=z_0 \\
	\left[
	\begin{array}{c}
	u_1\\
	u_2 \\
	u_3
	\end{array}
	\right]&=
	\left[
	\begin{array}{ccc}
	1 & -h_1 & h_1^2-h_2 \\
	0 & 1 & -h_1 \\
	0 & 0 & 1
	\end{array}
	\right]
	\left[
	\begin{array}{c}
	y_1\\
	y_2 \\
	y_3
	\end{array}
	\right].
\end{align}
\end{subequations}
In which case,
\begdi
	F_{d_{\rm{Toep}}^{-1}}[y]=
		\left[
		\begin{array}{ccc}
		1 & F_{d_1^{-1}}[y] & F_{d_2^{-1}}[y]  \\
		0 & 1   & F_{d_1^{-1}}[y] \\
		0 & 0 & 1
		\end{array}
		\right],
\enddi
where $d^{-1}=(d_1^{-1},d_2^{-1})$ is generated by $(\tilde{g}_1,\tilde{g}_2,\tilde{g}_3,z_0,\tilde{h}_1,\tilde{h}_2)$ with $\tilde{g}_1:=g_1$, $\tilde{g}_2:=g_2-g_1h_1$, $\tilde{g}_3:=g_3-g_2h_1+g_1(h_1^2-h_2)$, $\tilde{h}_1=-h_1$ and $\tilde{h}_2=h_1^2-h_2$. If coordinate functions are defined as linear maps on $\re^2\langle\langle X \rangle\rangle$ by
\begdi
	a^i_\eta(d):=(d_i,\eta)=L_{g_\eta} h_i(z_0),\quad \eta\in X^\ast, \;\;i=1,2,
\enddi
and $S$ is defined as a mapping on $\allpoly$ seen as dual space of $\allseries$, so that
\begdi
	(S(a^i_\eta))(d):=(d_i^{-1},\eta)=L_{\tilde{g}_\eta}\tilde{h}_i(z_0),\quad \eta\in X^\ast, \;\;i=1,2,
\enddi
then the coordinates, i.e., coefficients of the inverse series are described compactly by the following polynomials:
\begin{subequations} \label{antipode1time}
\begin{align}
	S(a^1_\mathrm{e})&=-a^1_\mathrm{e} \\
	S(a^2_\mathrm{e})&=-a^2_\mathrm{e}+a^1_\mathrm{e}a^1_\mathrm{e} \\
	S(a^1_{x_1})&=-a^1_{x_1} \\
	S(a^1_{x_2})&=-a^1_{x_2}+a^1_{x_1}a^1_\mathrm{e} \\
	S(a^2_{x_1})&=-a^2_{x_1}+2a^1_{x_1}a^1_\mathrm{e} \\
	S(a^1_{x_3})&=-a^1_{x_3}+a^1_{x_2}a^1_\mathrm{e}-a^1_{x_1}a^1_\mathrm{e}a^1_\mathrm{e}+a^1_{x_1}a^2_\mathrm{e} \\
	S(a^2_{x_2})&=-a^2_{x_2}+2a^1_{x_2}a^1_\mathrm{e}-2a^1_{x_1}a^1_\mathrm{e}a^1_\mathrm{e}+a^2_{x_1}a^1_\mathrm{e} \\
	S(a^2_{x_3})&=-a^2_{x_3}+2a^1_{x_3}a^1_\mathrm{e}-2a^1_{x_2}a^1_\mathrm{e}a^1_\mathrm{e}
					+a^2_{x_2}a^1_\mathrm{e}-a^2_{x_1}a^1_\mathrm{e}a^1_\mathrm{e}+a^2_{x_1}a^2_\mathrm{e}+ \nonumber\\
					&\hspace*{0.2in}2a_{x_1}^1a_\mathrm{e}^1a_\mathrm{e}^1a_\mathrm{e}^1-
				2a^1_{x_1}a_\mathrm{e}^1a_\mathrm{e}^2  \\
				&\hspace*{0.08in}\vdots\nonumber
\end{align}
\end{subequations}
\end{example}

 It is not obvious in general whether the generators for the inverse series $d^{-1}_i$ will necessarily satisfy a linear recursion of the form \rref{eq:Pn-recursion}. This is contingent on whether the new vector fields $\tilde g_i$ are consistent with the grading on $\allseries$, that is, whether $\deg(\tilde g_i)=\deg(g_i)$, $i=1,\ldots,m$. The next theorem gives a sufficient condition under which the upper triangular Toeplitz structure of $H$ in \rref{eq:Toeplitz-matrix} guarantees this property.

\begin{theorem} \label{th:Toeplitz-inverse-realization}
Given any Toeplitz matrix of the form \rref{eq:Toeplitz-matrix} and a set of vector fields $g_i$, $i=1,\ldots,m$ with
$\deg(g_i)=i$, it follows that $\tilde{g}_i:=\sum_{j=1}^m g_{j}H^{-1}_{ji}$ has the property $\deg(\tilde g_i)=\deg(g_i)$
provided $\deg(h_i):=\deg(g_i)=i$, $i=1,\ldots,m-1$.
\end{theorem}

\begpr
First observe that
\begdi
	H^{-1}=\left(I+\sum_{i=1}^{m-1} h_i N^i\right)^{-1}
		    =\sum_{n=0}^{m-1} (-1)^n\left(\sum_{i=1}^{m-1} h_i N^i\right)^n,
\enddi
using the fact that $N^n=0$, $n\geq m$. Now applying the multinomial theorem gives
\begin{align} \label{eq:Hinv-sum}
	H^{-1}&=I+\sum_{j=1}^{m-1} \left[\sum_{k=1}^j  (-1)^k k! \sum_{k_1+k_2+\cdots +k_j=k \atop k_1+2k_2+\cdots+jk_{j}=j}
				 \frac{1}{k_1!\cdots k_{j}!} h_1^{k_1}\cdots h_{j}^{k_{j}}\right] N^j. \\
		    &=:I+\sum_{j=1}^{m-1} \tilde{h}_j N^j. \nonumber
\end{align}
This means that $\deg(\tilde{h}_j)= \deg(h_1^{k_1}\cdots h_j^{k_{j}})=k_1+2k_2+\cdots+jk_{j}=j$. Therefore, since $\tilde{g}_i=\sum_{j=1}^m g_jH^{-1}_{ji}=\sum_{j=0}^i g_{i-j}\tilde{h}_j$ and
\begdi
	\deg\left(g_{i-j}\tilde{h}_j\right)
		=\deg\left(g_{i-j}\right)+\deg(\tilde{h}_j)
		=(i-j)+j
		=i,
\enddi
it follows that $\deg(\tilde g_i)=i$, $i=1,\ldots,m$ as required.
\endpr

\begin{example}
Reconsider Example~\ref{ex:mequal3-Toeplitz-affine-inverse-realization} in the particular case where
$(g_1,g_2,g_3,z_0,h_1,h_2)=(z^2,0,0,1,-z,0)$ so that $d=(-c_F,0)$.
This is an embedding of Example~\ref{ex:Ferfera-recursion} into the case where $m=3$.
The series $d^{-1}=(d_1^{-1},d_2^{-1})$ has the
generator $(\tilde{g}_1,\tilde{g}_2,\tilde{g}_3,z_0,\tilde{h}_1,\tilde{h}_2)=(z^2,z^3,z^4,1,z,z^2)$.
The system \rref{eq:mequal3-Toeplitz-affine-inverse-realization} reduces to the Abel system
\begin{align*}
\dot{z}&= z^2y_1+z^3y_2+z^4y_3,\quad z(0)=1 \\
\left[
\begin{array}{c}
u_1\\
u_2 \\
u_3
\end{array}
\right]&=
\left[
\begin{array}{ccc}
1 & z & z^2 \\
0 & 1 & z \\
0 & 0 & 1
\end{array}
\right]
\left[
\begin{array}{c}
y_1\\
y_2 \\
y_3
\end{array}
\right],
\end{align*}
and therefore, $c_{A,m}=d^{-1}_1$.
Using \rref{eq:Pn-recursion} with the generator for $d_1^{-1}$,
the Abel generating series $c_{A,3}$ can also be written as $c_{A,3}=\sum_{n\geq 1} c_{A,3}(n)$, where $c_{A,3}(1)=P_0(1)=\mathrm{e}$ and
\begdi
	c_{A,3}(n)=P_{n-1}(1)=L_{z^{2}}P_{n-2}(1)x_1+L_{z^{3}}P_{n-3}(1)x_2+L_{z^{4}}P_{n-4}(1)x_3,\quad n\geq 2
\enddi
($P_n(1):=0$ for $n<0$). A polynomial recursion follows from proving the identity $L_{z^{i+1}}P_{n-i-1}(1)=(n-i)P_{n-i-1}(1)$, $i=1,2,3$, so that
\begdi
	c_{A,3}(n)=(n-1)c_{A,3}(n-1)x_1+(n-2)c_{A,3}(n-2)x_2+(n-3)c_{A,3}(n-3)x_3,\quad n\geq 2
\enddi
($c_{A,3}(n)=0$ for $n<1$). The first few of these polynomials are:
\begin{subequations} \label{eq:cA3n}
\begin{align}
	c_{A,3}(1)&=1\\
	c_{A,3}(2)&=x_1 \\
	c_{A,3}(3)&=2x_1x_1+x_2 \\
	c_{A,3}(4)&=6x_1x_1x_1+3x_2x_1+2x_1x_2+x_3 \\
	c_{A,3}(5)&=24x_1x_1x_1x_1+12x_2x_1x_1+8x_1x_2x_1+4x_3x_1+6x_1x_1x_2+3x_2x_2+2x_1x_3.
\end{align}
\end{subequations}
Note that each $c_{A,3}(n)$ consists only of words of degree $n-1$. These polynomials were first identified by Devlin in \cite{Devlin_1989}.
The example can be generalized to any $m\geq 2$ so that
\begeq \label{eq:cAm-as-cF-inverse}
	c_{A,m}=(I-c_FN)^{-1}_1,
\endeq
and the corresponding Abel series $c_{A,m}=\sum_{n\geq 1} c_{A,m}(n)$ can be computed
from the recursion
\begdi
	c_{A,m}(n)=\sum_{i=1}^m (n-i)c_{A,m}(n-i)x_i,\quad n\geq 2, 
\enddi
with $c_{A,m}(1)=1$ and $c_{A,m}(n)=0$ for $n<1$.
\end{example}

It is interesting to note that the construction above has some elements in common with the Fa\`{a} di Bruno Hopf algebra ${\mathcal H}_{FdB}=(\mu,\Delta_{FdB})$ for the group $\mathcal{G}_{diff}$ of diffeomorphisms $h$ on $\re$ satisfying $h(0)=0$, $\dot{h}(0)=1$. See  \cite{Figueroa-Gracia-Bondia_05} for details. First observe that \rref{eq:Hinv-sum} can also be written in terms of the Bell polynomials
\begdi
	B_{j,k}(t_1,\ldots,t_l):=\sum_{k_1+k_2+\cdots +k_{l}=k \atop k_1+2k_2+\cdots+lk_{l}=j}
	\frac{j!}{k_1!\cdots k_{l}!} \left(\frac{t_1}{1!}\right)^{k_1}\cdots \left(\frac{t_l}{l!}\right)^{k_{l}},
\enddi
where $l=j-k+1$, using the Fa\`{a} di Bruno formula
\begdi
	f(h(t))=\sum_{j=0}^\infty \sum_{k=1}^j \beta_k B_{j,k}(\alpha_1,\ldots,\alpha_{j-k+1})\frac{t^j}{j!}
\enddi
with $f(t):=\sum_{n=0}^\infty \beta_n t^n/n!$ and $h(t)=\sum_{n=0}^\infty \alpha_n t^n/n!$. Specifically, setting
\begin{align*}
	 f(t)&=\frac{1}{1+t}=1+\sum_{n=1}^\infty (-1)^n n!\frac{t^n}{n!} \\
	h(t)&=\sum_{n=1}^{m-1} n!h_n \frac{t^n}{n!}
\end{align*}
gives
\begdi
	H^{-1}=f(h(N))=I+\sum_{j=1}^{m-1} \sum_{k=1}^j (-1)^k \left[\frac{k!}{j!} B_{j,k}(h_1,2! h_2,\ldots,(j-k+1)!h_{j-k+1})\right] N^j.
\enddi
The expressions in brackets above, i.e., the {\em ordinary} Bell polynomials, are used in \cite{Brudnyi_xxBdSM} to define a variation (flipped/co-opposite version) of the coproduct $\bar{\Delta}_{FdB}$ on ${\mathcal H}_{FdB}$ (see equations (4.1)-(4.2) in this citation).
A faithful representation of the group $\mathcal{G}_{diff}$ is
\begin{align*}
	M_h&:=\left[\frac{k!}{j!} B_{j,k}(h_1,2! h_2,\ldots,(j-k+1)!h_{j-k+1})\right]^T \\
		&=\left[\begin{array}{cccccc}
 			h_1 & h_2 & h_3 & h_4 & h_5 & \cdots \\
 			   0 & h_1^2 & 2 h_1 h_2 & h_2^2+2 h_1 h_3 & 2 h_2 h_3+2 h_1 h_4 & \cdots\\
 			   0 & 0 & h_1^3 & 3 h_1^2 h_2 & 3 h_1 h_2^2+3 h_1^2 h_3 & \cdots\\
 			   0 & 0 & 0 & h_1^4 & 4 h_1^3 h_2 & \cdots\\
 			   0 & 0 & 0 & 0 & h_1^5 & \cdots\\
 \vdots & \vdots & \vdots & \vdots & \vdots
\end{array}\right]
\end{align*}
(cf.~\cite{manchonfrabetti}) and
\begdi
	H^{-1}=I+\sum_{j=1}^{m-1} [\doubleone M_h]_j N^j=I+\sum_{j=1}^{m-1}\tilde{h}_jN^j,
\enddi
where $\doubleone:=[-1\;1\;-1\;\cdots]$. Therefore, defining the coordinate functions $a_i(h)=h_i$, $i\geq 1$, it follows that
$
	\tilde{h}_i= \mu(\bar{\Delta}_{FdB} a_i(\doubleone,h)).
$
For example,
\begin{align*}
	\tilde{h}_3	&=\mu(\bar{\Delta}_{FdB}(a_3)(\doubleone,h)) \\
				&=\mu((a_1\otimes a_3+a_2\otimes 2a_1a_2+a_3\otimes a_1^3)(\doubleone,h)) \\
				&=-h_3+2h_1h_2-h_1^3.
\end{align*}
Further observe, setting $h_1=1$, that the antipode $S_{FdB}$ of ${\mathcal H}_{FdB}$ can be identified from the top row of
\begdi
M_h^{-1}
	=\left[\begin{array}{cccccc}
 	1 & -h_2 & 2 h_2^2-h_3 & -5 h_2^3+5 h_2 h_3-h_4 & 14 h_2^4-21 h_2^2 h_3+3 h_3^2+6 h_2 h_4-h_5 & \cdots \\
 	0 & 1 & -2 h_2 & 5 h_2^2-2 h_3 & -14 h_2^3+12 h_2 h_3-2 h_4 & \cdots \\
 	0 & 0 & 1 & -3 h_2 & 9 h_2^2-3 h_3 & \cdots \\
 	0 & 0 & 0 & 1 & -4 h_2 & \cdots \\
 	0 & 0 & 0 & 0 & 1 & \cdots \\
 	\vdots & \vdots & \vdots & \vdots & \vdots
\end{array}\right].
\enddi
That is, using a standard expression for $S_{FdB}$ (see \cite{Figueroa-Gracia-Bondia_05}), the $(j+1)th$ entry in the top row of $M_h^{-1}$ is given by
\begin{align*}
	(S_{FdB}(a_{j+1}))(h)&=\sum_{k=1}^{j} (-1)^k B_{j+k,k}(0,2!h_2,3!h_3,\ldots,(j+1)!h_{j+1}),\quad j\geq 1.
\end{align*}
The assertion to be explored in Section~\ref{sect:multiHA} is that this construction has deeper connections to another kind of Fa\`{a} di Bruno type Hopf algebra, one that is derived from a group of Fliess operators and used to described their feedback interconnection. In fact, the compositional inverse described above corresponds to the group inverse.


\section{Devlin's polynomials and feedback recursions}
\label{sect:feedback-recursions}

It was shown in \cite{Ebrahimi-Fard-Gray_IMRN17} that the Devlin polynomials describing the Abel generating series $c_{A,m}$ when $m=2$ can be related to a certain feedback structure commonly encountered in control theory. This in turn led to a Hopf algebra interpretation of these polynomials since feedback systems have been characterized in such terms in \cite{DuffautEFG_2014,Gray-Duffaut_Espinosa_SCL11,Gray-Duffaut_Espinosa_FdB14,Gray-et-al_MTNS14,Gray-et-al_SCL14}.
In this section, the generalization of the theory is given for any $m\geq 2$. This will again provide a Hopf algebra interpretation of Devlin's polynomials as well as a shuffle formula for the Abel series which is distinct from that derived directly from the Abel equation, namely, the non-linear recursion
\begdi
	c_{A,m}=1+\sum_{i=1}^m x_ic_{A,m}^{\shuffle\,i+1},\quad m\geq 2,
\enddi
where $c_{A,m}^{\shuffle\,i}$ denotes the $i$-th shuffle power of $c_{A,m}$. Recall that $\allseries$ consisting of  all formal power series over the alphabet $X$ with coefficients in $\mathbb{R}$ forms an unital associative $\re$-algebra under the concatenation product and an unital, commutative and associative $\re$-algebra under the shuffle product, denoted here by the shuffle symbol $\shuffle$. The
latter is the $\re$-bilinear extension of the shuffle product of two words, which is defined inductively by
\begin{equation}
\label{shuffleproduct}
	(x_i\eta) \shuffle (x_j\xi)=x_i(\eta\shuffle(x_j\xi))+x_j((x_i\eta)\shuffle \xi)
\end{equation}
with $\eta\shuffle\emptyset=\emptyset\shuffle\eta=\eta$ for all words
$\eta,\xi\in X^\ast$ and letters $x_i,x_j\in X$ \cite{Fliess_81,reutenauer}.
For instance, $x_i \shuffle x_j = x_ix_j + x_jx_i$ and
\begin{equation*}
	x_{i_1} x_{i_2}\shuffle x_{i_3} x_{i_4} = x_{i_1} x_{i_2}x_{i_3} x_{i_4}
	+ x_{i_3} x_{i_4}x_{i_1} x_{i_2} + x_{i_1} x_{i_3}(x_{i_2}\shuffle x_{i_4})
	+ x_{i_3} x_{i_1}(x_{i_2}\shuffle x_{i_4}).
\end{equation*}

\begin{figure}[tb]
\begin{center}
\includegraphics[scale=0.35]{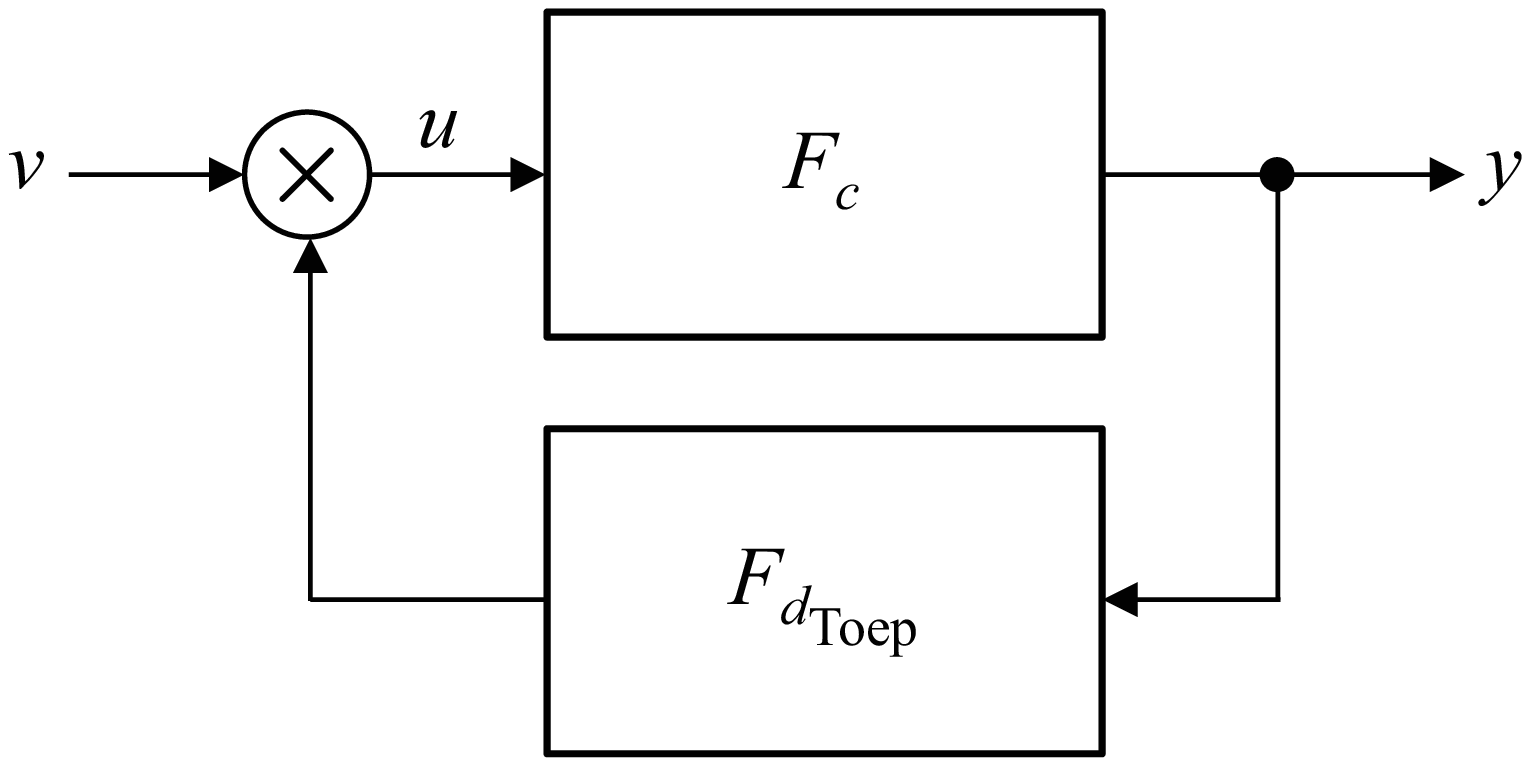}
\end{center}
\caption{Multiplicative feedback system}
\label{fig:multiplicative-feedback}
\end{figure}

Consider the (componentwise) multiplicative feedback interconnection shown in Figure~\ref{fig:multiplicative-feedback} consisting of a Fliess operator $F_c$ in the forward path, where $c\in\allseriesm$, and an $m\times m$ matrix-valued Toeplitz Fliess operator $F_{d_{\rm{Toep}}}$ in the feedback path.  It is useful here to define a {\em generalized unital series}, $\delta_i$, so that $F_{\delta_i}[y]:=y_i$ for all $i=1,\ldots,m$ and
$F_{\delta_i^{\shuffle j}}[y]:=(F_{\delta_i}[y])^j=y_i^j$. With $d=(\delta_1,\delta_1^{\shuffle 2},\ldots,\delta_1^{\shuffle m-1})$, the closed-loop system shown in
Figure~\ref{fig:multiplicative-feedback} is described by
\begdi
	y=F_c[u],\quad u=F_{d_{\rm{Toep}}}[y]v=v+\sum_{i=1}^{m-1} y_1^iN^i v,
\enddi
and, in particular,
\begeq
	u_1=v_1+y_1 v_2 + y_1^2 v_3+\cdots +y_1^{m-1}v_m. \label{eq:polynomial-feedback}
\endeq

\begex
Suppose $c_{F,m}=[c_F,0,\ldots,0]\in\allseriesm$,
where $c_F=\sum_{k=0}^\infty k!\,x_1^k$ as in Example~\ref{ex:Ferfera-recursion}.
In which case, $y_1=F_{c_F}[u]$ is realized by the one dimensional
state space model
\begdi
	\dot{z}=z^2u_1, \quad z(0)=1,\quad y_1=z.
\enddi
Applying the feedback \rref{eq:polynomial-feedback} gives the following
realization for the closed-loop system
\begdi
	\dot{z}=\sum_{i=1}^m v_i z^{i+1},\quad  z(0)=1,\;\;y_1=z, \quad i=1,\ldots,m.
\enddi
Hence, the generating series for the closed-loop system, denoted here by $c_{F,m}@d_{\rm{Toep}}$, has
the property that
\begeq \label{eq:c_Am-from-feedback}
	c_{A,m}=[c_{F,m}@d_{\rm{Toep}}]_1.
\endeq
This is a generalization of the result given in
\cite{Ebrahimi-Fard-Gray_IMRN17} for the $m=2$ case.
\endex

In control theory, feedback is often described algebraically in terms of transformation groups. This approach is useful here
as it will lead to an explicit way to compute the generating series of any closed-loop system as shown in Figure~\ref{fig:multiplicative-feedback}.
Consider the group of
Toeplitz affine Fliess operators
\begdi
	{\mathcal T}:=\left\{y=F_{d_\delta}[u]
			 =F_{d_{\rm{Toep}}}[u]u : d\in\re^{m-1}\langle\langle X\rangle\rangle\right\}
\enddi
under the operator composition
\begdi
	(F_{c_\delta}\circ F_{d_\delta})[u]
	=F_{c_{\rm{Toep}}}[F_{d_{\rm{Toep}}}[u]u]F_{d_{\rm{Toep}}}[u]u,
\enddi
which is associative and has the identity element $F_{0_\delta}$. Strictly speaking, one should limit the definition of the group to those generating series whose corresponding Fliess operators converge. But the algebraic set up presented here carries through in general if one considers the non-convergent case in a formal sense (see \cite{Gray-Wang_MTNS08}). The group inverse has already been described for the case where $d$ is differentially generated, i.e., by equation \rref{eq:inverse-affine-FO}. It can be shown by other arguments to exist in general (via contractive maps on ultrametric spaces, see \cite{Gray-Ebrahimi-Fard_SIAMzz}). The group product on ${\mathcal T}$ in turn induces a formal power series product on $\re^{m-1}\langle\langle X\rangle\rangle$ denoted by $c_\delta\circ d_\delta$ satisfying
$
	F_{c_\delta\circ d_\delta}=F_{c_\delta}\circ F_{d_\delta}.
$
Given that generating series are unique and the bijection between $\re^{m-1}\langle\langle X\rangle\rangle$ and their associated Toeplitz matrices, this means that $\re^{m-1}\langle\langle X\rangle\rangle$ inherits a group structure. A right action of the group $\mathcal T$ on the set of all Fliess operators $F_c$, $c\in\allseriesm$ is given by
$$
	(F_c\circ F_{d_\delta})[u]=F_{c}[F_{d_{\rm{Toep}}}[u]u]
					 =F_c\left[u+\sum_{i=1}^{m-1} F_{d_i}[u]N^iu\right].
$$
This composition induces a second formal power series product, the {\em mixed composition product} $c\modcomp d_\delta$, satisfying
\begin{equation} \label{eq:modcomp-product}
	F_c\circ F_{d_\delta}=F_{c\modcomp d_{\delta}}.
\end{equation}
It can be viewed as a right action of the group $\re^{m-1}\langle\langle X\rangle\rangle$ on the set $\allseriesm$. This product is left linear, nonassociative, and can be computed explicitly when $c\in\allseries$ by
\begin{displaymath}
	c \modcomp d_\delta
	=\phi_d(c)(\mathbf{1})
	=\sum_{\eta\in X^\ast} \langle c,\eta\rangle \, \phi_d(\eta)(\mathbf{1}),
\end{displaymath}
where $\mathbf{1}:=1\mathrm{e}$, and $\phi_d$ is the continuous (in the ultrametric sense) algebra homomorphism from $\allseries$ to $\Endallseries$ uniquely specified by $\phi_d(x_i\eta)=\phi_d(x_i)\circ \phi_d(\eta)$ with
\begeq \label{eq:phi-definition}
	\phi_d(x_i)(e)=x_ie+\sum_{j=1}^{m-i} x_{i+j}(d_j\shuffle e),\quad i=1,\ldots,m
\endeq
for any $e\in\allseries$, and where $\phi_d(\emptyset)$ denotes the identity map on $\allseries$. For any  $c\in\re^{i\times j}\langle\langle X\rangle\rangle$ the product is extended componentwise
such that
\begeq \label{eq:modcomp-product2}
	[c\modcomp d]_{kl}=[c]_{kl}\modcomp d
\endeq
for all $k=1,2,\ldots,i$ and $l=1,2,\ldots,j$. The following pre-Lie product results from the right linearization of the mixed composition product
\begdi
\label{pre-Lie}
	x_i\eta\ \triangleleft\ d := x_i(\eta\ \triangleleft\ d) + \sum_{j=1}^{m-i}x_{i+j}(d_j \shuffle \eta)
\enddi
with $\emptyset\triangleleft d:=0$. In which case, $c \modcomp d_\delta = c + c \triangleleft d + O(d^2)$. In particular, it can be shown directly that
\begdi
	(c_{\delta}\circ d_\delta)_{\rm{Toep}}= (c_{\rm{Toep}}\modcomp d_\delta)\shuffle d_{\rm{Toep}},
\enddi
where the shuffle product on matrix-valued series is defined componentwise. Another useful composition product is the (unmixed) composition product $c\circ d$ induced simply by $F_{c\circ d}=F_c\circ F_d$.

With these various formal power series products defined, it is now possible to give a general formula for the {\em feedback product} $c@d_{\rm{Toep}}$ describing the generating series for the interconnected system in Figure~\ref{fig:multiplicative-feedback}. The following lemma is needed.

\begin{lemma} \label{le:MIMO-shuffle-group}
The set $\mathcal{G}_{\shuffle}:=\{c\in \allseriesmtimesm: \langle c,\mathrm{e}\rangle\in Gl_m(\re)\}$ is a group under the
shuffle product with the identity element being the constant series $\mathbf{I}:=I\mathrm{e}$, and the inverse of any
$c\in \mathcal{G}_{\shuffle}$ is
\begdi
	c^{\shuffle -1}	=(\langle c,\mathrm{e}\rangle(\mathbf{I}-c^\prime))^{\shuffle -1}
				=(c^\prime)^{\shuffle \ast}\langle c,\mathrm{e}\rangle^{-1},
\enddi
where $c^\prime$ is proper (i.e, $\langle c^\prime,\mathrm{e}\rangle=0$), and $(c^\prime)^{\shuffle\ast}:=\sum_{k\geq 0}
(c^\prime)^{\shuffle k}$.
\end{lemma}

\begin{theorem} \label{th:c-at-d-general-formula}
For any $c\in\allseriesm$ and $d\in\allseriesmminusone$ it follows that $c@d_{\rm{Toep}}=c\modcomp((d_{\rm{Toep}}\circ c)^{\shuffle -1})_\delta^{-1}$.
\end{theorem}

\begpr
The feedback law requires that $u=F_{d_{\rm{Toep}}}[y]v=F_{d_{\rm{Toep}}}[F_c[u]]v=F_{d_{\rm{Toep}}\circ c}[u]v$.
From Lemma~\ref{le:MIMO-shuffle-group} it follows that
\begdi
	v=F_{(d_{\rm{Toep}}\circ c)^{\shuffle -1}}[u]u=F_{((d_{\rm{Toep}}\circ c)^{\shuffle -1})_\delta}[u].
\enddi
As the latter is now a group element in ${\mathcal T}$, one can write
\begdi
	u=F_{((d_{\rm{Toep}}\circ c)^{\shuffle -1})_\delta^{-1}}[v].
\enddi
Making this substitution for $u$ into $y=F_c[u]$ and writing the result in terms of the group action gives
\begdi
	y=F_{c@d_{\rm{Toep}}}[v]
	  =F_c[F_{((d_{\rm{Toep}}\circ c)^{\shuffle -1})_\delta^{-1}}[v]]
 	  =F_{c\modcomp((d_{\rm{Toep}}\circ c)^{\shuffle -1})_\delta^{-1}}[v].
\enddi
As generating series are known to be unique, the theorem is proved.
\endpr

\begin{corollary} \label{co:feedback-fixed-point}
The feedback product satisfies the fixed point equation $c@d_{\rm{Toep}}=c\modcomp(d_{\rm{Toep}}\circ (c@d_{\rm{Toep}}))_\delta$.
\end{corollary}

\begpr
Observe that
$
	y=F_c[F_{d_{\rm{Toep}}}[y]v]
$. So if $y=F_{c@d_{\rm{Toep}}}[v]$ then necessarily
\begdi
	y=F_c[F_{d_{\rm{Toep}}}[F_{c@d_{\rm{Toep}}}[v]]v]
	 =F_c[F_{d_{\rm{Toep}}\circ (c@d_{\rm{Toep}})}[v]v]
	 =F_{c\modcomp(d_{\rm{Toep}}\circ(c@d_{\rm{Toep}}))_\delta}[v].
\enddi
The uniqueness of generating series then proves the claim.
\endpr

The tools above are now applied to compute the feedback product $c_{F,m}@d_{\rm{Toep}}$ in \rref{eq:c_Am-from-feedback}.
This will in turn render identities satisfied by the Abel series. The following lemma is useful.

\begin{lemma} \label{lem:inverse-rational-Toeplitz-matrix}
If in (\ref{eq:Toeplitz-matrix}) $h_i=h^i$, $i=1,\ldots,m-1$ for some $h\in C^\omega$ then $H^{-1}=I-hN$.
\end{lemma}

\begpr
Given that
$
	H=I+h N+h^2N^2+\cdots +h^{m-1}N^{m-1},
$
observe
\begin{align*}
	H^{-1}
	&=((I-hN)^{-1}-(hN)^m(I-hN)^{-1})^{-1} \\
	&=(I-hN)(I-(hN)^m)^{-1} \\
	&=I-hN+O((hN)^m) \\
	&=I-hN,
\end{align*}
since $N^n=0$ when $n\geq m$.
\endpr

\begin{theorem} \label{th:cAm-feedback-formula}
For any $m\geq 2$, $c_{A,m}=c_F\modcomp(I-c_FN)^{-1}_\delta$.
\end{theorem}

\begpr
Starting from the formula in Theorem~\ref{th:c-at-d-general-formula}
for the feedback product with $c=c_{F,m}=[c_F,0,\ldots,0]$ and $d=(\delta_1,\delta_1^{\shuffle 2},\ldots,\delta_1^{\shuffle m-1})$
and using the definition of the shuffle inverse in Lemma~\ref{le:MIMO-shuffle-group}, observe that
\begin{align*}
	c_{F,m}@d_{\rm{Toep}}&=c_{F,m}\modcomp((d_{\rm{Toep}}\circ c_{F,m})^{\shuffle -1})_\delta^{-1} \\
	&=c_{F,m}\modcomp\left(\left(\left(I+\sum_{i=1}^{m-1} \delta_1^{\shuffle i} N^i\right)\circ c_{F,m}\right)^{\shuffle -1}\right)_\delta^{-1} \\
	&=c_{F,m}\modcomp\left(\left(\sum_{i=0}^{m-1}c_F^{\shuffle i}N^i\right)^{\shuffle -1}\right)_\delta^{-1}.
\end{align*}
Now note that if $h$ in Lemma~\ref{lem:inverse-rational-Toeplitz-matrix} is identified with $F_{c_F}$
then $h^i=F^i_{c_F}=F_{c_F^{\shuffle i}}$. So the shuffle version of the identity in this lemma is
$\left(\sum_{i=0}^{m-1}c_F^{\shuffle i}N^i\right)^{\shuffle -1}=I-c_FN$. In which case,
$
	c_{F,m}@d_{\rm{Toep}}=c_{F,m}\modcomp (I-c_FN)_\delta^{-1}.
$
Next, in light of \rref{eq:c_Am-from-feedback} and \rref{eq:modcomp-product2}, it is clear that
$c_{A,m}=[c_{F,m}\modcomp (I-c_FN)_\delta^{-1}]_1=c_F\modcomp (I-c_FN)_\delta^{-1}$ as claimed.
\endpr

\begin{example}
Consider evaluating $c_{A,m}=c_F\modcomp(I-c_FN)^{-1}_\delta$ when $m=3$. In this case
\begdi
c_F\modcomp(I-c_FN)^{-1}_\delta=\sum_{k=0}^\infty k!\; \phi_d(x_1^k)(\mathbf{1}),
\enddi
where
\begdi
\phi_{d}(x_1)(e)=x_1e+x_2(d_1\shuffle e)+x_3(d_2\shuffle e)
\enddi
with $d_1=(I-c_FN)^{-1}_1$ and $d_2=(I-c_FN)^{-1}_2$. Using \rref{antipode1time} to compute the inverses gives
\begin{align*}
	\langle d_1,\mathrm{e}\rangle&=S(a^1_\mathrm{e})(-c_{F})=-a^1_\mathrm{e}(-c_{F})=-\langle-c_F,\mathrm{e}\rangle=1 \\
	\langle d_1,x_1\rangle&=S(a^1_{x_1})(-c_{F})=-a^1_{x_1}(-c_{F})=-\langle-c_F,x_1\rangle=1 \\
	\langle d_1,x_2\rangle&=S(a^1_{x_2})(-c_{F})=(-a^1_{x_2}+a^1_{x_1}a^1_\mathrm{e})(-c_{F})\\
                 &=-\langle-c_F,x_2\rangle+\langle-c_F,x_1\rangle\langle -c_F,\mathrm{e}\rangle=1\\
	\langle d_1,x_3\rangle&=S(a^1_{x_3})(-c_{F})=(-a^1_{x_3}+a^1_{x_2}a^1_\mathrm{e}-a^1_{x_1}a^1_\mathrm{e}a^1_\mathrm{e}+a^1_{x_1}a^2_\mathrm{e})(-c_{F}) \\
                 &=-\langle -c_F,x_3\rangle+\langle -c_F,x_2\rangle\langle -c_F,\mathrm{e}\rangle-\langle-c_F,x_1\rangle\langle-c_F,\mathrm{e}\rangle^2+
                 \langle -c_F,x_1\rangle\langle 0,\mathrm{e}\rangle \\
                 &=1.
\end{align*}
Therefore, $d_1=1+x_1+x_2+x_3+\cdots$, which from \rref{eq:cAm-as-cF-inverse} should be $c_{A,3}$. Similarly, $d_2=1+2x_1+2x_2+2x_3$, so that
\begin{align*}
c_F\modcomp(I-c_FN)^{-1}_\delta&=1 + x_1 + x_2 + x_3 + 2 x_1x_1 +2 x_1x_2 + 2 x_1x_3 + 3 x_2x_1 +3 x_2x_2 \\
&\hspace*{0.18in}+3 x_2x_3 +4 x_3x_1+4 x_3x_2 + 4 x_3x_3 + \cdots,
\end{align*}
which is also equivalent to $c_{A,3}$ as expected.
\end{example}

\begin{theorem} \label{th:c_Am-shuffle-equation}
For any $m\geq 2$
\begdi
	c_{A,m}=1+c_{A,m}\shuffle\left(\sum_{i=1}^m x_i c_{A,m}^{\shuffle i-1}\right).
\enddi
\end{theorem}

\begpr
Applying Corollary~\ref{co:feedback-fixed-point}, Theorem~\ref{th:cAm-feedback-formula}, and the fact that the mixed composition product distributes to the left over the shuffle product gives
\begin{align*}
	c_{A,m}&=c_F\modcomp (d_{\rm Toep}\circ c_{A,m})_\delta
           =c_F\modcomp \left(\sum_{i=0}^m c_A^{\shuffle i}N^i\right)_\delta \\
           &=\sum_{k=0}^\infty x_1^{\shuffle k}\modcomp \left(\sum_{i=0}^m c_A^{\shuffle i}N^i\right)_\delta
		   =\sum_{k=0}^\infty \left(x_1\modcomp  \left(\sum_{i=0}^m c_A^{\shuffle i}N^i\right)_\delta \right)^{\shuffle k} \\
   	  	   &=\sum_{k=0}^\infty \left(\sum_{i=1}^m x_i c_{A,m}^{\shuffle i-1}\right)^{\shuffle k}.
\end{align*}
Hence, the identity in question then follows directly.
\endpr

Theorem~\ref{th:c_Am-shuffle-equation} was first observed in functional form for the $m=2$ case in \cite{Lijun-Yun_01} (see equation (2.3)). In fact, one of the main results of this paper (Theorem 4.1) is actually just a graded version of this result as described next.

\begin{corollary}
For any $m,n\geq 2$
\begdi
	c_{A,m}(n)=c_{A,m}(n-1)\shuffle x_1+\sum_{i=2}^m\sum_{k_1+\cdots+k_i=n-1}
	c_{A,m}(k_1)\shuffle (x_i(c_{A,m}(k_2)\shuffle\cdots c_{A,m}(k_i))).
\enddi
\end{corollary}

\begin{example}
When $m=3$ observe
$
	c_{A,3}=1+c_{A,3}\shuffle(x_1+x_2c_{A,3}+x_3c_{A,3}^{\shuffle 2}).
$
Therefore, if $a:=F_{c_{A,3}}$ then
\begdi
	a(t)=1+a(t)\left[\int_0^t u_1(\tau)\,d\tau+\int_0^t u_2(\tau)a(\tau)\,d\tau+\int_0^t u_3(\tau)a^2(\tau)\,d\tau\right].
\enddi
Defining $a_n=F_{c_{A,3}(n)}$, $n\geq 1$ gives the recursion
\begin{align*}
	a_n(t)&=a_{n-1}(t)\int_0^t u_1(\tau)\,d\tau+\sum_{k_1+k_2=n-1}
			a_{k_1}(t)\int_0^t u_2(\tau)a_{k_2}(\tau)\,d\tau+ \\
		&\hspace*{0.2in}\sum_{k_1+k_2+k_3=n-1}
		a_{k_1}(t)\int_0^t u_3(\tau)a_{k_2}(\tau)a_{k_3}(\tau)\,d\tau.
\end{align*}
The $m=2$ case of this recursion appears in \cite{Lijun-Yun_01} as equation (1.7).
\end{example}

The final theorem will be generalized in Section~\ref{sect:center-condition} to provide a sufficient condition for a center of the Abel equation.

\begin{theorem} \label{th:Abel-solution-r-equals-1}
Let $v_1,v_2,\ldots,v_m\in L_1[0,\omega]$ and $m\geq 2$ be fixed. Then the $m+1$ degree Abel
equation~(\ref{eq:Abel-eqn-degree-m}) with $z(0)=1$
has the solution
\begdi
	z(t)=\frac{1}{1-E_{x_1}[u](t)},
\enddi
if there exists functions $u_1,u_2,\ldots,u_m\in L_1[0,\omega]$ satisfying
\begin{align*}
	v_1(t)&=u_1(t)-\frac{u_2(t)}{1-E_{x_1}[u](t)} \\
	v_2(t)&=u_2(t)-\frac{u_3(t)}{1-E_{x_1}[u](t)} \\
		&\hspace*{0.1in}\vdots \\
	v_{m-1}(t)&=u_{m-1}(t)-\frac{u_m(t)}{1-E_{x_1}[u](t)} \\
	v_m(t)&=u_m(t),
\end{align*}
with $E_{x_1}[u](t):=\int_0^t u_1(\tau)\,d\tau<1$ on $[0,\omega]$.
\end{theorem}

\begpr
In light of Theorem~\ref{th:cAm-feedback-formula}, it is clear that
$
	c_{A,m}=c_F\modcomp (I-c_FN)^{-1}_\delta,
$
and thus,
$
	c_F=c_{A,m}\modcomp (I-c_FN)_\delta.
$
So assume there exists $u\in L^m_1[0,\omega]$ such that
$$
	v=F_{(I-c_F N)_\delta}[u]
	=\left[
	\begin{array}{c}
	u_1-F_{c_F}[u]u_2 \\
	u_2-F_{c_F}[u]u_3 \\
	\vdots \\
	u_{m-1}-F_{c_F}[u]u_m \\
	u_m
	\end{array}
	\right].
$$
Then, observing that $F_{c_F}[u]=1/(1-E_{x_1}[u])$, it follows from (\ref{eq:modcomp-product}) that
$$
	z(t)=F_{c_{A,m}}[v]
	     =F_{c_{A.m}}[F_{(I-c_F N)_\delta}[u]]=F_{c_{A,m}\modcomp (I-c_FN)_\delta}[u]
	     =F_{c_F}[u]=\frac{1}{1-E_{x_1}[u](t)}.
$$
\endpr

In the next section a Hopf algebra structure is defined on the coordinate functions.


\section{Multivariable Hopf algebra for Toeplitz multiplicative output feedback}
\label{sect:multiHA}

All algebraic structures considered in this section are over the field $\mathbb{K}$ of characteristic zero. Let $X=\{x_1,\ldots,x_m\}$ be a finite alphabet with $m$ letters. Each letter has an integer degree $\deg(x_k):=k$. The monoid of words is denoted by $X^*$ and includes the empty word $\mathrm{e}=\emptyset$ for which $\deg(\mathrm{e})=0$. The degree of a word $\eta = x_{i_1} \cdots x_{i_n} \in X^*$ of length $|\eta |:=n$ is defined by
$$
	\deg(\eta):=\sum_{k=1}^{m} k|\eta|_k.
$$
Here $|\eta|_k$ denotes the number of times the letter $x_k \in X$ appears in the word $\eta$.

Consider the polynomial algebra $H^{(\bar{m})}$ generated by the {\it{coordinate functions}} $a^k_\eta$, where $\eta \in X^*$ and the so called {\it{root index}} $k \in [\bar{m}]:=\{1, \ldots, \bar{m}\}$, $\bar{m}\le m$. By defining the degree
\begdi
	\|a^{k}_\eta\| := k + \deg(\eta),
\enddi
$H^{(\bar{m})}$ becomes a graded connected algebra, $H^{(\bar{m})}:=\bigoplus_{n \ge 0} H^{(\bar{m})}_n$, and $\|a^{k}_\eta a^{l}_\kappa\| = \|a^{k}_\eta\| + \|a^{l}_\kappa\|$. The unit in $H^{(\bar{m})}$ is denoted by $\mathbf{1}$, and $\|\mathbf{1}\|=0$, whereas $\|a^{k}_\mathrm{e}\| = k$.

The {\it{left-}} and {\it{right-shift}} maps, ${\theta}_{x_j}: H^{(\bar{m})} \to H^{(\bar{m})}$ respectively $\tilde{\theta}_{x_j}: H^{(\bar{m})} \to H^{(\bar{m})}$, for $x_{j} \in X$, are
taken to be
$$
	{\theta}_{x_j} a^p_\eta:=a^p_{x_j\eta},
	\qquad\
	\tilde{\theta}_{x_{j}} a^p_\eta:= {a}^p_{\eta x_{j}}
$$
with $\theta_{x_j}\un = \tilde{\theta}_{x_j}\un =0$. On products in $H^{(\bar{m})}$ both these maps act as derivations
$$
	{\theta}_{x_j} a^p_\eta a^q_\mu
	:=({\theta}_{x_j} a^p_\eta) a^q_\mu +  a^p_\eta ({\theta}_{x_j}a^q_\mu),
$$
and analogously for $\tilde{\theta}_{x_{j}}$. For a word $\eta = x_{i_1} \cdots x_{i_n} \in X^*$
$$
	{\theta}_{\eta} := {\theta}_{x_{i_1}} \cdots {\theta}_{x_{i_n}},
	\qquad\
	\tilde{\theta}_{\eta} := \tilde{\theta}_{x_{i_n}} \cdots \tilde{\theta}_{x_{i_1}}.
$$
Hence, any element $a^i_\eta \in H^{(\bar{m})}$ with $\eta = x_{i_1} \cdots x_{i_n} \in X^*$ can be written
$$
	a^i_\eta = {\theta}_{\eta}  a_\mathrm{e}^i = \tilde{\theta}_{\eta} a_\mathrm{e}^i.
$$

In the following it will be shown how $\tilde{\theta}_{\eta}$ can be employed to define a coproduct $\Delta : H^{(\bar{m})} \to H^{(\bar{m})} \otimes H^{(\bar{m})}$. First, for the coordinate functions with respect to the empty word, $a_\mathrm{e}^l$, $1 \leq l \leq \bar{m}$, the coproduct is defined to be
\begin{align}
\label{coproduct1a}
	\Delta a_\mathrm{e}^l := a_\mathrm{e}^l \otimes \un + \un \otimes a_\mathrm{e}^l
						+ \sum_{k=1}^{l-1} a_\mathrm{e}^k \otimes a_\mathrm{e}^{l-k}.
\end{align}
Note that $a_\mathrm{e}^1$ is by definition primitive, i.e., $\Delta a_\mathrm{e}^1 = a_\mathrm{e}^1 \otimes \un + \un \otimes a_\mathrm{e}^1$. The next step is to define $\Delta$ on any $a^i_\eta$ with $1 \leq i \leq \bar{m}$ and $|\eta|>0$ by specifying intertwining relations between the maps $\tilde{\theta}_{x_i}$ and the coproduct
\begin{equation}
\label{cproduct1b}
	\Delta \circ {\tilde\theta}_{x_i}
	:= \left(\tilde{\theta}_{x_i} \otimes \id + \id \otimes \tilde{\theta}_{x_i}
	+ \sum_{j=1}^{i-1} \tilde{\theta}_{x_j} \otimes A_\mathrm{e}^{(i-j)}\right) \circ \Delta.
\end{equation}
The map $A_\mathrm{e}^{(k)}$ is defined by
$$
	A_\mathrm{e}^{(k)} a^i_\eta := a^i_\eta a_\mathrm{e}^{k}.
$$
The following notation is used, $\Delta \circ {\tilde\theta}_{x_i}={\tilde\Theta}_{x_i} \circ \Delta$, where
\begdi
	{\tilde\Theta}_{x_i} :=  \tilde{\theta}_{x_i} \otimes \id + \id \otimes \tilde{\theta}_{x_i}
	+ \sum_{j=1}^{i-1} \tilde{\theta}_{x_j} \otimes A_\mathrm{e}^{(i-j)},
\enddi
and ${\tilde\Theta}_{\eta}:={\tilde\Theta}_{x_{i_n}}\cdots {\tilde\Theta}_{x_{i_1}}$ for $\eta = x_{i_1} \cdots x_{i_n} \in X^*$. In this setting, $a^{1}_{x_1}$ is primitive since
$$
	\Delta a^{1}_{x_1}
	= \Delta \circ {\tilde\theta}_{x_1} a^{1}_{\mathrm{e}}
	={\tilde\Theta}_{x_1} \circ \Delta a^{1}_{\mathrm{e}}
	= (\tilde{\theta}_{x_1} \otimes \id + \id \otimes \tilde{\theta}_{x_1})( a_\mathrm{e}^1 \otimes \un + \un \otimes a_\mathrm{e}^1)
	= a^{1}_{x_1} \otimes \un + \un \otimes a^{1}_{x_1},
$$
which follows from $ \tilde{\theta}_{x_j} \un = 0$. The coproduct of $a^{l}_{x_2}$ is
\allowdisplaybreaks
\begin{align}
	\lefteqn{\Delta a^{l}_{x_2}  = \Delta \circ \tilde{\theta}_{x_2} a_\mathrm{e}^{l}
		= {\tilde\Theta}_{x_2} \circ \Delta a_\mathrm{e}^{l}
		=\big(\tilde{\theta}_{x_2} \otimes \id + \id \otimes \tilde{\theta}_{x_2}
	+ \tilde{\theta}_{x_1} \otimes A_\mathrm{e}^{(1)}\big) \circ \Delta a_\mathrm{e}^{l} }	\nonumber\\
			&= a^{l}_{x_2} \otimes \un + \un \otimes a^{l}_{x_2}
			+ a^{l}_{x_1} \otimes a^{1}_{\mathrm{e}} + \sum_{k=1}^{l-1} a_{x_2}^k \otimes a_\mathrm{e}^{l-k}
			  + \sum_{k=1}^{l-1} a_\mathrm{e}^k \otimes a_{x_2}^{l-k}
			  + \sum_{k=1}^{l-1} a_{x_1}^k \otimes a_\mathrm{e}^1a_\mathrm{e}^{l-k}. \nonumber
\end{align}
The coproduct of a general $a^{l}_{x_i}$ is
\allowdisplaybreaks
\begin{align}
	\Delta a^{l}_{x_i} &= \Delta \circ \tilde{\theta}_{x_i} a_\mathrm{e}^{l}
		= {\tilde\Theta}_{x_i} \circ \Delta a_\mathrm{e}^{l}
		=\big(\tilde{\theta}_{x_i} \otimes \id + \id \otimes \tilde{\theta}_{x_i}
	+ \sum_{j=1}^{i-1} \tilde{\theta}_{x_j} \otimes A_\mathrm{e}^{(i-j)}\big)
	\circ \Delta a_\mathrm{e}^{l}	\nonumber\\
			&= a^{l}_{x_i} \otimes \un + \un \otimes a^{l}_{x_i}
			+  \sum_{j=1}^{i-1} a^{l}_{x_j} \otimes a^{i-j}_{\mathrm{e}} + \sum_{k=1}^{l-1} a_{x_i}^k \otimes a_\mathrm{e}^{l-k}
            + \sum_{k=1}^{l-1} a_\mathrm{e}^k \otimes a_{x_i}^{l-k}\nonumber\\	
			&\hspace*{0.2in}+  \sum_{j=1}^{i-1}\sum_{k=1}^{l-1} a_{x_j}^k \otimes a_\mathrm{e}^{i-j}a_\mathrm{e}^{l-k}. \label{coprodA}
\end{align}
Observe that the grading is preserved. A few examples may be helpful
\allowdisplaybreaks
\begin{align*}
	\Delta' a^{2}_{x_1} &=
			 a^{1}_{x_1} \otimes a^{1}_{\mathrm{e}}
			+ a^{1}_{\mathrm{e}}  \otimes a^{1}_{x_1}				\\
	\Delta' a^{1}_{x_2} &=
			 a^{1}_{x_1} \otimes a^{1}_{\mathrm{e}} 				\\
	\Delta' a^{2}_{x_2} &=
					 a^{2}_{x_1} \otimes a^{1}_{\mathrm{e}} 	
				    	+ a_{x_2}^1 \otimes a_\mathrm{e}^{1}
			  		+ a_\mathrm{e}^1 \otimes a_{x_2}^{1}
			  		+ a_{x_1}^1 \otimes a_\mathrm{e}^{1}a_\mathrm{e}^{1}. \\
	\Delta' a^{1}_{x_3} &=
					 a^{1}_{x_1} \otimes a^{2}_{\mathrm{e}}
					+ a^{1}_{x_2} \otimes a^{1}_{\mathrm{e}}			\\
	\Delta' a^{2}_{x_3} &=
					 a^{2}_{x_1} \otimes a^{2}_{\mathrm{e}}
					+ a^{2}_{x_2} \otimes a^{1}_{\mathrm{e}}		
					+ a_{x_3}^1 \otimes a_\mathrm{e}^{1}
			  + a_\mathrm{e}^1 \otimes a_{x_3}^{1}
			  + a_{x_1}^1 \otimes a_\mathrm{e}^{2}a_\mathrm{e}^{1}
			  + a_{x_2}^1 \otimes a_\mathrm{e}^{1}a_\mathrm{e}^{1},
\end{align*}
where $\Delta' a^{l}_{\eta}:=\Delta a^{l}_{\eta} - a^{l}_{\eta} \otimes \un - \un \otimes a^{l}_{\eta}$ is the reduced coproduct. For the element $a^{l}_{x_2x_1}$ one finds the following coproduct
\allowdisplaybreaks
\begin{align}
		\lefteqn{\Delta a^{l}_{x_2x_1}
		= \Delta \circ \tilde{\theta}_{x_1} \tilde{\theta}_{x_2} a_\mathrm{e}^{l}
		= {\tilde\Theta}_{x_1}{\tilde\Theta}_{x_2} \circ \Delta a_\mathrm{e}^{l}}\nonumber\\
			&=a^{l}_{x_2x_1} \otimes \un + \un \otimes a^{l}_{x_2x_1}
				+ a^{l}_{x_1} \otimes a^{1}_{x_1}
				+ a^{l}_{x_1x_1} \otimes a^{1}_{\mathrm{e}}
				+ \sum_{k=1}^{l-1} a_{x_2x_1}^k \otimes a_\mathrm{e}^{l-k}\nonumber\\
			&	+ \sum_{k=1}^{l-1} a_{x_1}^k \otimes a_{x_2}^{l-k}
				+ \sum_{k=1}^{l-1} a_{x_1x_1}^k \otimes a_\mathrm{e}^1a_\mathrm{e}^{l-k}
				+ \sum_{k=1}^{l-1} a_{x_2}^k \otimes a_{x_1}^{l-k}
			  	+ \sum_{k=1}^{l-1} a_\mathrm{e}^k \otimes a_{x_2x_1}^{l-k}\nonumber\\	
			& 	+ \sum_{k=1}^{l-1} a_{x_1}^k \otimes a_{x_1}^1a_\mathrm{e}^{l-k}
			  	+ \sum_{k=1}^{l-1} a_{x_1}^k \otimes a_\mathrm{e}^1a_{x_1}^{l-k}.\nonumber
\end{align}
The general formula for words of length two is
\allowdisplaybreaks
\begin{align*}
		\lefteqn{\Delta a^{l}_{x_jx_i}
		= \Delta \circ \tilde{\theta}_{x_i} \tilde{\theta}_{x_j} a_\mathrm{e}^{l}
		= {\tilde\Theta}_{x_i} {\tilde\Theta}_{x_j} \circ \Delta a_\mathrm{e}^{l}	}  \\
			& = a^{l}_{x_jx_i} \otimes \un + \un \otimes a^{l}_{x_jx_i}
				+  \sum_{n=1}^{j-1} a^{l}_{x_nx_i} \otimes a^{j-n}_{\mathrm{e}}
				+  \sum_{n=1}^{j-1} a^{l}_{x_n} \otimes a^{j-n}_{x_i}
				+ \sum_{s=1}^{i-1} a^{l}_{x_jx_s} \otimes a_\mathrm{e}^{i-s} \\
				&+ \sum_{s=1}^{i-1} \sum_{n=1}^{j-1} a^{l}_{x_nx_s} \otimes a_\mathrm{e}^{i-s} a^{j-n}_{\mathrm{e}}
				+ \sum_{k=1}^{l-1} a_{x_jx_i}^k \otimes a_\mathrm{e}^{l-k}
			 	+ \sum_{k=1}^{l-1} a_{x_i}^k \otimes a_{x_j}^{l-k}
			 	+  \sum_{n=1}^{j-1}\sum_{k=1}^{l-1} a_{x_nx_i}^k \otimes a_\mathrm{e}^{j-n}a_\mathrm{e}^{l-k}	\\
			&	+ \sum_{k=1}^{l-1} a_{x_j}^k \otimes a_{x_i}^{l-k}
			  	+ \sum_{k=1}^{l-1} a_\mathrm{e}^k \otimes a_{x_jx_i}^{l-k}
			 	+  \sum_{n=1}^{j-1}\sum_{k=1}^{l-1} a_{x_n}^k \otimes a_{x_i}^{j-n}a_\mathrm{e}^{l-k}
			  	+  \sum_{n=1}^{j-1}\sum_{k=1}^{l-1} a_{x_n}^k \otimes a_\mathrm{e}^{j-n}a_{x_i}^{l-k}\\
			&	+  \sum_{s=1}^{i-1}\sum_{k=1}^{l-1} a_{x_jx_s}^k \otimes a_\mathrm{e}^{i-s}a_\mathrm{e}^{l-k}
				+  \sum_{s=1}^{i-1}\sum_{k=1}^{l-1} a_{x_s}^k \otimes a_\mathrm{e}^{i-s}a_{x_j}^{l-k}
				+  \sum_{s=1}^{i-1} \sum_{n=1}^{j-1}\sum_{k=1}^{l-1} a_{x_nx_s}^k \otimes a_\mathrm{e}^{i-s}a_\mathrm{e}^{j-n}a_\mathrm{e}^{l-k}.
\end{align*}
The coproduct is then extended multiplicatively to all of $H^{(\bar{m})}$ and $\Delta(\un):=\un \otimes \un$.

\begin{theorem}\label{thm:newMIMO-HA}
$H^{(\bar{m})}$ is a connected graded commutative non-cocommutative Hopf algebra with unit map $u:\mathbb{K} \to H^{(\bar{m})} $, counit $\epsilon: H^{(\bar{m})} \to \mathbb{K}$ and coproduct $\Delta: H^{(\bar{m})} \to H^{(\bar{m})} \otimes H^{(\bar{m})}$
\begin{equation}
\label{coproduct}
	\Delta a_\eta^k ={\tilde\Theta}_{\eta}\circ \Delta  a^k_\mathrm{e}.
\end{equation}
\end{theorem}

\begin{proof}
$H^{(\bar{m})}=\bigoplus_{n\ge0} H_n^{(\bar{m})}$ is connected graded and commutative by construction. In addition, it is clear that the coproduct is non-cocommutative. What is left to be shown is coassociativity. This is done by first proving the claim for $a_\mathrm{e}^{l}$, which follows from the identity
$$
	\sum_{k=1}^{l-1} \sum_{p=1}^{k-1} a_\mathrm{e}^p \otimes a_\mathrm{e}^{k-p} \otimes a_\mathrm{e}^{l-k}
	=
	\sum_{k=1}^{l-1} \sum_{p=1}^{l-k-1} a_\mathrm{e}^{k} \otimes a_\mathrm{e}^p \otimes a_\mathrm{e}^{l-k-p}.
$$
From $ \Delta(a_{\eta x_i}^k) = \Delta \circ \tilde\theta_{x_{i}} (a_{\eta}^k)
					 = \tilde\Theta_{x_i} \circ  \Delta(a_{\eta}^k)$ it follows that
\begin{align*}
	\lefteqn{(\Delta \otimes \id) \circ \Delta(a_{\eta x_i}^k) = (\Delta \otimes \id)  \circ \tilde \Theta_{x_i} \circ \Delta(a_{\eta}^k)}\\
	&= \Big( \Delta  \circ \tilde{\theta}_{x_i} \otimes \id +  \id \otimes \id \otimes \tilde{\theta}_{x_i}
	+ \sum_{j=1}^{i-1} \Delta \circ  \tilde{\theta}_{x_j} \otimes A_\mathrm{e}^{(i-j)}\Big) \circ \Delta(a_{\eta}^k)\\
	&=\Big(\tilde{\Theta}_{x_i} \otimes \id +  \id \otimes \id \otimes \tilde{\theta}_{x_i}
	+ \sum_{j=1}^{i-1} \tilde{\Theta}_{x_j} \otimes A_\mathrm{e}^{(i-j)}\Big) (\Delta \otimes \id)  \circ \Delta(a_{\eta}^k)\\
	&=\Big(\tilde{\theta}_{x_i} \otimes \id \otimes \id
		+ \id \otimes \tilde{\theta}_{x_i} \otimes \id
		+  \id \otimes \id \otimes \tilde{\theta}_{x_i} \\
	&	+ \sum_{j=1}^{i-1} \tilde{\theta}_{x_j} \otimes A_\mathrm{e}^{(i-j)}\otimes \id
	 	+ \sum_{j=1}^{i-1} \tilde{\theta}_{x_j} \otimes \id \otimes A_\mathrm{e}^{(i-j)} \\
	&    + \sum_{j=1}^{i-1}  \id \otimes \tilde{\theta}_{x_j} \otimes A_\mathrm{e}^{(i-j)}
	     + \sum_{j=1}^{i-1} \sum_{k=1}^{j-1} \tilde{\theta}_{x_k} \otimes A_\mathrm{e}^{(j-k)} \otimes A_\mathrm{e}^{(i-j)}\Big)
		(\id \otimes \Delta)  \circ \Delta(a_{\eta}^k)\\
	&= \Big(\tilde{\theta}_{x_i} \otimes \id \otimes \id + \id \otimes \tilde{\Theta}_{x_i}
		+ \sum_{j=1}^{i-1} \tilde{\theta}_{x_j} \otimes \Big(A_\mathrm{e}^{(i-j)}\otimes \id  + \id \otimes A_\mathrm{e}^{(i-j)}\Big)\\	&	+ \sum_{j=1}^{i-1} \sum_{k=1}^{j-1} \tilde{\theta}_{x_k} \otimes A_\mathrm{e}^{(j-k)} \otimes A_\mathrm{e}^{(i-j)}\Big)
		(\id \otimes \Delta)  \circ \Delta(a_{\eta}^k).
\end{align*}
As noted above, the last sum can be rewritten as
$$
	\sum_{j=1}^{i-1} \sum_{k=1}^{j-1} \tilde{\theta}_{x_k} \otimes A_\mathrm{e}^{(j-k)} \otimes A_\mathrm{e}^{(i-j)}
	=
	\sum_{j=1}^{i-2} \sum_{k=1}^{i-j-1} \tilde{\theta}_{x_j} \otimes A_\mathrm{e}^{(k)} \otimes A_\mathrm{e}^{(i-j-k)}
$$
so that
\begin{align*}
	\lefteqn{\Big(\tilde{\theta}_{x_i} \otimes \id \otimes \id + \id \otimes \tilde{\Theta}_{x_i}
		+ \sum_{j=1}^{i-1} \tilde{\theta}_{x_j} \otimes \big(A_\mathrm{e}^{(i-j)}\otimes \id  + \id \otimes A_\mathrm{e}^{(i-j)}\big)}\\	&	+ \sum_{j=1}^{i-1} \sum_{k=1}^{j-1} \tilde{\theta}_{x_k} \otimes A_\mathrm{e}^{(j-k)} \otimes A_\mathrm{e}^{(i-j)}\Big)
		(\id \otimes \Delta)  \circ \Delta(a_{\eta}^k)\\
	&= \Big(\tilde{\theta}_{x_i} \otimes \id \otimes \id + \id \otimes \tilde{\Theta}_{x_i} \\
	&	+ \sum_{j=1}^{i-1} \tilde{\theta}_{x_j} \otimes \Big(A_\mathrm{e}^{(i-j)}\otimes \id  + \id \otimes A_\mathrm{e}^{(i-j)}
		+ \sum_{k=1}^{i-j-1}  A_\mathrm{e}^{(k)} \otimes A_\mathrm{e}^{(i-j-k)}\Big)\Big)
		(\id \otimes \Delta)  \circ \Delta(a_{\eta}^k)\\
	&= (\id \otimes \Delta)\circ
	\Big(\tilde{\theta}_{x_i} \otimes \id + \id \otimes \tilde{\theta}_{x_i}
		+ \sum_{j=1}^{i-1} \tilde{\theta}_{x_j} \otimes A_\mathrm{e}^{(i-j)}\Big) \circ \Delta(a_{\eta}^k)\\
	&= (\id \otimes \Delta)\circ\Delta(a_{\eta x_i}^k).
\end{align*}
The following was also used in the calculation above
$$
	\Delta \circ A_\mathrm{e}^{(i-j)}
	=
	\Big(A_\mathrm{e}^{(i-j)}\otimes \id  + \id \otimes A_\mathrm{e}^{(i-j)}
		+ \sum_{k=1}^{i-j-1}  A_\mathrm{e}^{(k)} \otimes A_\mathrm{e}^{(i-j-k)}\Big)\circ\Delta,
$$
which follows from $A_\mathrm{e}^{(l)} a_{\eta}^k = a_\mathrm{e}^l a_{\eta}^k$ together with the multiplicativity of $\Delta$.
\end{proof}

In the following, a variant of Sweedler's notation \cite{Sweedler_69} is used for the reduced coproduct, i.e., $\Delta'(a^{l}_{\eta})= \sum a^{l'}_{\eta'} \otimes a^{l''}_{\eta''}$, as well as for the full coproduct
$$
	\Delta(a^{l}_{\eta}) = \sum a^{l'}_{\eta_{(1)}} \otimes a^{l''}_{\eta_{(2)}}
				      = a^{l}_{\eta} \otimes {\bf{1}} + {\bf{1}} \otimes a^{l}_{\eta} + \Delta'(a^{l}_{\eta}).
$$
Connectedness of $H^{(\bar{m})}$ implies for the antipode $S: H^{(\bar{m})} \to H^{(\bar{m})}$ the well known recursions
\begin{equation}
\label{antipode1}
	S a^{l}_{\eta} = -a^{l}_{\eta} - \sum S(a^{l'}_{\eta'})a^{l''}_{\eta''}
			     = -a^{l}_{\eta} - \sum a^{l'}_{\eta'}S(a^{l''}_{\eta''}).
\end{equation}
A few examples are given first. Coproduct \eqref{coproduct1a} implies for the elements $a^{k}_\mathrm{e}$ that\begin{equation}
\label{emptywordcoprod}
	S a^{k}_\mathrm{e} = -a^{k}_\mathrm{e}
	+ \sum_{i=2}^{k} (-1)^i \sum_{p_1 + \cdots + p_i = k \atop p_j >0} a^{p_1}_\mathrm{e} \cdots a^{p_i}_\mathrm{e}.
\end{equation}
For example,
\begdi
	S a^{1}_\mathrm{e} = - a^{1}_{\mathrm{e}},			
	\qquad S a^{2}_\mathrm{e} = - a^{2}_{\mathrm{e}} + a^{1}_{\mathrm{e}}a^{1}_{\mathrm{e}},
	\qquad S a^{3}_\mathrm{e} = - a^{3}_{\mathrm{e}} + 2 a^{1}_{\mathrm{e}}a^{2}_{\mathrm{e}}
	- a^{1}_{\mathrm{e}}a^{1}_{\mathrm{e}}a^{1}_{\mathrm{e}}.		
\enddi
The following examples are given for comparison with \eqref{antipode1time}:
\begin{align*}	
	S a^{1}_{x_1} &= - a^{1}_{x_1}							\\
	S a^{1}_{x_2} &= - a^{1}_{x_2}
					+ a^{1}_{x_1} a^{1}_{\mathrm{e}} 		\\
	S a^{1}_{x_3} &= - a^{1}_{x_3}
					+ a^{1}_{x_1} a^{2}_{\mathrm{e}}
					- a^{1}_{x_1} a^{1}_{\mathrm{e}} a^{1}_{\mathrm{e}}
					+ a^{1}_{x_2}a^{1}_{\mathrm{e}}		\\
	S a^{2}_{x_1} &= - a^{2}_{x_1}	
					+ 2 a^{1}_{x_1}a^{1}_{\mathrm{e}} 		\\
	S a^{2}_{x_2} &= - a^{2}_{x_2}
					+ a^{2}_{x_1}a^{1}_{\mathrm{e}}	
					- 2 a^{1}_{x_1}a^{1}_{\mathrm{e}} a^{1}_{\mathrm{e}}
					+ 2 a^{1}_{x_2} a^{1}_{\mathrm{e}}	\\
	S a^{2}_{x_3} &= - a^{2}_{x_3}
					+ 2 a^{1}_{x_3} a^{1}_{\mathrm{e}}	
					- 2 a^{1}_{x_2} a^{1}_{\mathrm{e}}a^{1}_{\mathrm{e}}
					+ a^{2}_{x_2} a^{1}_{\mathrm{e}}
					- a^{2}_{x_1}a^{1}_{\mathrm{e}}a^{1}_{\mathrm{e}}
					+ a^{2}_{x_1}a^{2}_{\mathrm{e}} \nonumber\\
					&\hspace*{0.2in}- 2 a^{1}_{x_1}a^{1}_{\mathrm{e}} a^{2}_{\mathrm{e}}
					+ 2 a^{1}_{x_1}a^{1}_{\mathrm{e}} a^{1}_{\mathrm{e}}a^{1}_{\mathrm{e}}
\end{align*}

The next theorem uses the coproduct formula \eqref{cproduct1b} to provide a simple formula for the antipode of $H^{(\bar{m})}$.

\begin{theorem}\label{thm:antipode}
For any nonempty word $\eta= x_{i_1} \cdots x_{i_l}\in X^\ast$, the antipode $S: H^{(\bar{m})} \to H^{(\bar{m})}$ can be written as
\begin{equation}
\label{antipodeSpecial}
	S a^k_\eta =  \tilde\Theta'_{\eta}(Sa^k_\mathrm{e}),
\end{equation}
where $\tilde\Theta'_{\eta}:= \tilde\theta'_{x_{i_l} } \circ \cdots \circ  \tilde\theta'_{x_{i_1}}$
with
\begin{align*}
	\tilde\theta'_{x_l}&:=  \tilde \theta_{x_l}
				+ \sum_{j=1}^{l-1} S(a_\mathrm{e}^{l-j}) \tilde{\theta}_{x_j}.
\end{align*}
\end{theorem}

For instance, calculating
\begin{align*}
	\tilde\Theta'_{x_1}(Sa^3_\mathrm{e})
	&= \tilde \theta_{x_1}(- a^{3}_{\mathrm{e}} + 2 a^{1}_{\mathrm{e}}a^{2}_{\mathrm{e}}
	- a^{1}_{\mathrm{e}}a^{1}_{\mathrm{e}}a^{1}_{\mathrm{e}})= - a^{3}_{x_1}
		+ 2 a^{1}_{x_1}a^{2}_{\mathrm{e}}
		+ 2 a^{1}_{\mathrm{e}}a^{2}_{x_1}
		- 3 a^{1}_{x_1}a^{1}_{\mathrm{e}}a^{1}_{\mathrm{e}},
\end{align*}
which coincides with $Sa^3_{x_1}$. Another example is
$$
	\tilde\Theta'_{x_2}(Sa^2_\mathrm{e})
	= (  \tilde \theta_{x_2}
				+ S(a_\mathrm{e}^{1}) \tilde{\theta}_{x_1})( - a^{2}_{\mathrm{e}} + a^{1}_{\mathrm{e}}a^{1}_{\mathrm{e}})
	=  - a^{2}_{x_2}
	    + 2 a^{1}_{x_2} a^{1}_{\mathrm{e}}
					+ a^{2}_{x_1}a^{1}_{\mathrm{e}}	
					- 2 a^{1}_{x_1}a^{1}_{\mathrm{e}} a^{1}_{\mathrm{e}}.			
$$

\begin{proof}
The proof follows by a nested induction using the weight of the root index and word length. First, formula \eqref{antipodeSpecial} is shown to hold for words of length one. Note that the recursions \eqref{antipode1} can be written in terms of the convolution product, i.e., $-S=P * S=S * P$, which is defined in terms of the coproduct \eqref{coproduct}
$$
	S=-m_{H^{(\bar{m})}}\circ\big(P \otimes S \big)\circ\Delta
	  =-m_{H^{(\bar{m})}}\circ\big( S \otimes P \big)\circ\Delta.
$$
Here $m_{H^{(\bar{m})}}$ denotes the product in $H^{(\bar{m})}$ and $P:= \id - u \circ \epsilon$ is the projector that maps the unit $\un$ in ${H^{(\bar{m})}}$ to zero and reduces to the identity on $H_+^{(\bar{m})}=\bigoplus_{n>0} H_n^{(\bar{m})}$.  Formula \eqref{antipodeSpecial} applied to $a^1_{x_l}$ gives
$$
	 \tilde\Theta'_{x_l}(Sa^1_\mathrm{e})
	 = \Big(\tilde \theta_{x_l}
				+ \sum_{j=1}^{l-1} S(a_\mathrm{e}^{l-j}) \tilde{\theta}_{x_j}\Big)Sa^1_\mathrm{e}
	 =-a^1_{x_l} - \sum_{j=1}^{l-1} S(a_\mathrm{e}^{l-j}) a^1_{x_j}, 			
$$
where \eqref{emptywordcoprod} was used. This coincides with
\begin{align*}
	Sa^1_{x_l}
	&=-m_{H^{(\bar{m})}}\circ\big(P \otimes S \big) \circ \Delta a^1_{x_l}\\
	&=-m_{H^{(\bar{m})}}\circ\big(P \otimes S \big)\Big(
	a^{1}_{x_l} \otimes \un + \un \otimes a^{1}_{x_l}
			+  \sum_{j=1}^{l-1} a^{1}_{x_j} \otimes a^{l-j}_{\mathrm{e}}\Big) \nonumber\\	
	&=-a^1_{x_l} - \sum_{j=1}^{l-1} S(a_\mathrm{e}^{l-j}) a^1_{x_j}. 		
\end{align*}
Now \eqref{antipodeSpecial} applied to $a^k_{x_l}$ gives
\begin{align*}
	 \tilde\Theta'_{x_l}(Sa^k_\mathrm{e})
	 &= \Big(\tilde \theta_{x_l}
				+ \sum_{j=1}^{l-1} S(a_\mathrm{e}^{l-j}) \tilde{\theta}_{x_j}\Big)Sa^k_\mathrm{e}\\		
          &= \Big(\tilde \theta_{x_l}
				+ \sum_{j=1}^{l-1} S(a_\mathrm{e}^{l-j}) \tilde{\theta}_{x_j}\Big)
				\Big(-a^k_\mathrm{e} - \sum_{w=1}^{k-1}a^w_\mathrm{e}Sa^{k-w}_\mathrm{e} \Big)\\	
	 &=-a^k_{x_l} - \sum_{j=1}^{l-1} S(a_\mathrm{e}^{l-j}) a^k_{x_j}
	 - \sum_{w=1}^{k-1}a^w_{x_l}Sa^{k-w}_\mathrm{e}
	 - \sum_{j=1}^{l-1}\sum_{w=1}^{k-1} a^w_{x_j}S(a_\mathrm{e}^{l-j})S(a^{k-w}_\mathrm{e}) \\
	 &\hspace*{0.2in}- \sum_{w=1}^{k-1}a^w_\mathrm{e} \tilde\Theta'_{x_l}Sa^{k-w}_\mathrm{e}.	
\end{align*}
Using the induction hypothesis on the last term, namely, $\tilde\Theta'_{x_l}Sa^{k-w}_\mathrm{e} = Sa^{k-w}_{x_l}$, gives
\begin{align*}
	\tilde\Theta'_{x_l}(Sa^k_\mathrm{e})
	&= -a^k_{x_l} - \sum_{j=1}^{l-1} S(a_\mathrm{e}^{l-j}) a^k_{x_j}
	    - \sum_{w=1}^{k-1}a^w_{x_l}Sa^{k-w}_\mathrm{e}
	    - \sum_{j=1}^{l-1}\sum_{w=1}^{k-1} a^w_{x_j}S(a_\mathrm{e}^{l-j})S(a^{k-w}_\mathrm{e}) \\
	&\hspace*{0.2in}  - \sum_{w=1}^{k-1}a^w_\mathrm{e} Sa^{k-w}_{x_l}.
\end{align*}
This coincides with the antipode computed via the coproduct in \eqref{coprodA} since
\begin{align*}
	Sa^k_{x_l}
	&=-m_{H^{(\bar{m})}}\circ\big(P \otimes S \big) \circ \Delta a^k_{x_l}\\
	&=-m_{H^{(\bar{m})}}\circ\big(P \otimes S \big)\Big(
	a^{k}_{x_l} \otimes \un + \un \otimes a^{k}_{x_l}
			+  \sum_{j=1}^{l-1} a^{k}_{x_j} \otimes a^{l-j}_{\mathrm{e}} \nonumber\\	
			&\hspace*{0.2in}+ \sum_{w=1}^{k-1} a_{x_l}^w \otimes a_\mathrm{e}^{k-w}
			  +  \sum_{j=1}^{l-1}\sum_{w=1}^{k-1} a_{x_j}^w \otimes a_\mathrm{e}^{l-j}a_\mathrm{e}^{k-w}
			  + \sum_{w=1}^{k-1} a_\mathrm{e}^w \otimes a_{x_l}^{k-w}
			  \Big).
\end{align*}
Now suppose \rref{antipodeSpecial} holds for all words $\nu \in X^*$ up to length $|\nu|=n-1$. The final step is to consider $a^l_\eta$, where $\eta=x_{i_1} \cdots x_{i_n}=\bar \eta x_{i_n}$, i.e., $|\eta|=n$, and $l \in [\bar{m}]$. Observe
\begin{align*}
	\tilde\Theta'_{\eta}(a^l_\mathrm{e})
	&= \tilde\Theta'_{x_{i_n}} S(a^l_{\bar\eta})\\
	&=  -\tilde\Theta'_{x_{i_n}} m_{H^{(\bar{m})}}\circ\big(P \otimes S \big)\circ\Delta a^l_{\bar\eta}\\
	&= -m_{H^{(\bar{m})}}\circ\big((\tilde\Theta'_{x_{i_n}} \otimes \id
	+ \id \otimes \tilde\Theta'_{x_{i_n}}) \circ (P\otimes S) \big)\circ\Delta a^l_{\bar\eta}\\
	&= -m_{H^{(\bar{m})}}\circ\big(
	   P \circ \tilde\Theta'_{x_{i_n}}  \otimes S
	+ P \otimes \tilde\Theta'_{x_{i_n}}\circ S \big)\circ\Delta a^l_{\bar\eta}\\
	&= -m_{H^{(\bar{m})}}\circ\Big(
	   P \circ \tilde \theta_{x_{i_n}} \otimes S
	+ P \circ \sum_{j=1}^{i_n-1} S(a_\mathrm{e}^{i_n-j}) \tilde{\theta}_{x_j} \otimes S
	+ P \otimes S \circ  \tilde \theta_{x_{i_n}}\Big)\circ\Delta a^l_{\bar\eta}\\
	&= -m_{H^{(\bar{m})}}\circ\Big(
	   (P \otimes S ) \circ \big(\tilde \theta_{x_{i_n}} \otimes \id
	+ \sum_{j=1}^{i_n-1} \tilde{\theta}_{x_j} \otimes A_\mathrm{e}^{(i_n-j)}
	+ \id \otimes \tilde \theta_{x_{i_n}}\big)\Big)\circ\Delta a^l_{\bar\eta}\\
	&= -m_{H^{(\bar{m})}}\circ\big(
	   P \otimes S  \big)\circ \tilde\Theta_{x_{i_n}} \circ \Delta a^l_{\bar\eta}\\
	&= -m_{H^{(\bar{m})}}\circ\big(
	   P \otimes S  \big)\circ \Delta a^l_{\eta}=Sa^l_{\eta}.
\end{align*}
The third equality above came from that fact the
$\tilde\Theta'_{x_{i_n}}$ is a sum of derivations. The fourth equality is a consequence of the identity $P \circ \tilde \theta_{x_{i_n}}=\tilde \theta_{x_{i_n}} \circ P$. The step from the fourth to the fifth equality used the induction hypothesis to get $P \otimes \tilde\Theta'_{x_{i_n}}\circ S = P \otimes S \circ  \tilde \theta_{x_{i_n}}$, which holds due to the projector $P$ being on the left-hand side. In addition, the following identity was used:
$$
	m_{H^{(\bar{m})}}\circ\Big(P \circ \sum_{j=1}^{i_n-1} S(a_\mathrm{e}^{i_n-j}) \tilde{\theta}_{x_j} \otimes S \Big)\circ\Delta
	= m_{H^{(\bar{m})}}\circ\Big((P \otimes S) \circ \sum_{j=1}^{i_n-1} \tilde{\theta}_{x_j} \otimes A_\mathrm{e}^{(i_n-j)} \Big)\circ\Delta,
$$
which holds due to $S$ being an algebra morphism.
\end{proof}

The final result is evident from the fact that the feedback structures in Figures~\ref{fig:output-feedback} and \ref{fig:multiplicative-feedback}
coincide when condition \rref{eq:polynomial-feedback} holds with $m=2$.

\begin{corollary}\label{cor:SISO}
For the alphabet $X:=\{x_1,x_2\}$ the Hopf algebra $H^{(1)}$ coincides with the Fa\`a di Bruno-type Hopf algebra for single-input, single-output (SISO) output feedback given in \cite{Foissy_13,Gray-Duffaut_Espinosa_SCL11,Gray-Duffaut_Espinosa_FdB14}.
\end{corollary}


\section{Sufficient condition for a center of the Abel equation}
\label{sect:center-condition}

Consider first a new sufficient condition for a center inspired by viewing the Abel equation in terms of a feedback
connection as described in Section~\ref{sect:feedback-recursions}.

\begin{theorem} \label{th:Abel-center}
Let $v_1,v_2,\ldots,v_m\in L_1[0,\omega]$ and $m\geq 2$ be fixed. Then the $m+1$ degree Abel
equation~(\ref{eq:Abel-eqn-degree-m}) has a center at $z=0$ if there exists an $R>0$ such that for every $r<R$ the system of equations
\begin{subequations} \label{eq:uv-equations}
\begin{align}
		v_1(t)&=u_1(t)-\frac{r u_2(t)}{1-rE_{x_1}[u](t)} \\
		v_2(t)&=u_2(t)-\frac{r u_3(t)}{1-rE_{x_1}[u](t)} \\
			&\hspace*{0.1in}\vdots \nonumber \\
	  v_{m-1}(t)&=u_{m-1}(t)-\frac{r u_m(t)}{1-rE_{x_1}[u](t)} \\
	         v_m(t)&=u_m(t),
\end{align}
\end{subequations}
has a solution $u_1,u_2,\ldots,u_m\in L_1[0,\omega]$ with $E_{x_1}[u](t):=\int_0^t u_1(\tau)\,d\tau <1/r$ on the interval $[0,\omega]$ and $E_{x_1}[u](\omega)=0$.
\end{theorem}

\begpr
The claim is proved by showing that if the system \rref{eq:uv-equations} has the solution $u_1$, $u_2$,\ldots, $u_m$ then the Abel equation (\ref{eq:Abel-eqn-degree-m}) with $z(0)=r<R$ has the solution
\begin{equation} \label{eq:Abel-equation-solution-r-leq-R}
	z(t)=\frac{r}{1-rE_{x_1}[u](t)}.
\end{equation}
In which case, $z(0)=z(\omega)=r$ for all $r<R$ so that $z=0$ is a center.

Consider the case where $m=2$ for simplicity. The proposed solution for \rref{eq:Abel-eqn-degree-m} can be checked by direct substitution. That is,
\begdi
	\dot{z}(t)=\frac{r^2}{\left(1-rE_{x_1}[u](t)\right)^2}u_1(t),
\enddi
so that
\begin{align*}
	v_1(t)z^2(t)+v_2(t)z^3(t)
	&=\left[u_1(t)-\frac{r u_2(t)}{1-rE_{x_1}[u](t)}\right]\left[\frac{r}{1-rE_{x_1}[u](t)}\right]^2+ \\
	&\hspace*{0.18in} u_2(t)\left[\frac{r}{1-rE_{x_1}[u](t)}\right]^3 \\
	&=\frac{r^2}{(1-rE_{x_1}[u](t))^2}u_1(t)
\end{align*}
as expected.
\endpr

Recall it was shown in Theorem~\ref{th:Abel-solution-r-equals-1} where $z(0)=1$ that $z(t)=F_{c_{A,m}}[v](t)=1/(1-E_{x_1}[u](t))$. So for sufficiently small $R>0$ and given any $r<R$ the solution to equation~\rref{eq:Abel-eqn-degree-m} with $z(0)=r$ can be written in the form
\begdi
	z(t)=r\sum_{n=1}^\infty F_{c_{A,m}(n)}[v](t)r^n
	=r\sum_{n=1}^\infty F_{r^nc_{A,m}(n)}[v](t)=:r\sum_{n=1}^\infty F_{c^\prime_{A,m}(n)}[v](t).
\enddi
So letting $c_{A,m}^\prime:=\sum_{n=1}^\infty c_{A,m}^\prime(n)$, the composition condition \rref{eq:composition-condition} ensures periodic solutions because
\begin{align*}
	z(\omega)&=rF_{c_{A.m}^\prime}[v](\omega)
	=r\sum_{\eta\in X^\ast} \langle c_{A,m}^\prime,\eta \rangle E_\eta[v](\omega) \\
		&=r\sum_{\eta\in X^\ast} \langle c_{A,m}^\prime,\eta \rangle E_\eta[\bar{v}](q(\omega))
		=r\sum_{\eta\in X^\ast} \langle c_{A,m}^\prime,\eta \rangle E_\eta[\bar{v}](q(0)) \\
		&=rE_{\emptyset}[\bar{v}](q(0))=r=z(0),
\end{align*}
using the fact that $E_\eta[\bar{v}](q(0))=0$ for all $\eta\neq \emptyset$. Put another way, the composition condition gives periodic solutions by simply ensuring that $E_{\eta}[v](\omega)=0$ for every nonempty word $\eta\in X^\ast$. In which case, it is immediate from the shuffle  identity $x_i^{\shuffle k}=k!x_i^k$
that the {\em moment conditions} with respect to $v$
\begdi
	\int_{0}^\omega v_i(\tau)E^k_{x_1}[v](\tau)\,d\tau
	=k!E_{x_ix_1^k}[v](\omega)=0, \quad i=2,3,\ldots,m,\;\;k\geq 0
\enddi
are satisfied. It is known for polynomial $v_i$, however, that the moment conditions do not imply the composition condition \cite{Gine-etal_16}. The following theorem indicates a condition under which the two conditions are satisfied with respect to the $u_i$ functions.

\begth
Suppose the $v_1,v_2,\ldots,v_m\in L_1[0,\omega]$ satisfy the composition condition. Let $u_1,u_2,\ldots,u_m\in L_1[0,\omega]$ be any solution to \rref{eq:uv-equations} with $E_{x_1}[u](t):=\int_0^t u_1(\tau)\,d\tau <1/r$ on the interval $[0,\omega]$. Then the composition condition and the moment conditions with respect to the $u_i$ are equivalent.
\endth

\begpr
Integrating both sides of \rref{eq:uv-equations} over $[0,\omega]$ gives
\begin{align*}
	E_{x_{i}}[v](\omega) &=E_{x_{i}}[u](\omega)-
	r\sum_{k=0}^\infty r^k \int_0^\omega u_{i+1}(t) E^k_{x_1}[u](\tau)\,d\tau \\
	=&E_{x_{i}}[u](\omega)-r\sum_{k=0}^\infty r^k k! E_{x_{i+1}x^k_1}[u](\omega)
\end{align*}
for $i=1,2,\ldots,m-1$ with $E_{x_m}[v](\omega)=E_{x_m}[u](\omega)$. Therefore, if the $v_i$ satisfy the composition condition then the left-hand side of this equation is zero. In which case, the claim follows immediately.
\endpr


\end{document}